\documentclass{amsart}

\newcommand{\al}{\alpha}
\newcommand{\ep}{\epsilon}
\newcommand{\vphi}{\varphi}

\newtheorem{theorem}{Theorem}[section]
\newtheorem{lemma}[theorem]{Lemma}
\newtheorem{claim}[theorem]{Claim}
\newtheorem{proposition}[theorem]{Proposition}
\newtheorem{corollary}[theorem]{Corollary}

\theoremstyle{definition}
\newtheorem{definition}[theorem]{Definition}

\theoremstyle{remark}
\newtheorem{remark}[theorem]{Remark}

\numberwithin{equation}{section}

\begin{document}

\title[An integro-PDE model]{An integro-PDE model \\ for evolution of random dispersal}


%    author one information
\author{King-Yeung Lam}
\address{Department of Mathematics, Columbus, OH, United States.}
\curraddr{}
\email{lam.184@math.ohio-state.edu}
\thanks{KYL and YL are partially supported by NSF grant DMS-1411476}

%    author two information
\author{Yuan Lou}
\address{Institute for Mathematical Sciences, Renmin University of China, and Department of Mathematics, Ohio State University}
\curraddr{}
\email{lou@math.ohio-state.edu}
%\thanks{Research partially supported by NSF grant DMS-1411476}

%    \subjclass is required.
\subjclass[2010]{35K57, 92D15,  92D25}

\date{\today}

\dedicatory{}

%    Abstract is required.
\begin{abstract}
We consider an  integro-PDE model for a population structured by the spatial variables and a trait variable which is the diffusion rate. Competition for resource is local in spatial variables, but nonlocal in the trait variable. We focus on the asymptotic profile of positive steady state solutions. Our result shows that in the limit of small mutation rate, the solution remains regular in the spatial variables and yet concentrates in the trait variable and forms a Dirac mass supported at the lowest diffusion rate. Hastings
and Dockery et al. showed that for two competing species  in spatially heterogeneous
but temporally constant environment, the slower diffuser always prevails, if all other things are held equal \cite{DHMP, Hastings}.
Our result suggests that their findings may well hold for arbitrarily many or even a continuum of  traits.

\end{abstract}

\maketitle

\section{Introduction}\label{sec:1}

In this paper, we focus on the concentration phenomena in a mutation-selection model for the evolution of random dispersal in a bounded, spatially heterogeneous and temporally constant environment. This model concerns a population structured simultaneously by a spatial variable $x \in D$ and the motility trait $\alpha \in \mathcal{A}$ of the species. Here $D$ is a bounded open domain in $\mathbb{R}^N$,
and $\mathcal{A} = [\underline\alpha, \overline\alpha]$, with $\overline{\alpha}>\underline\alpha>0$, denotes a bounded set of phenotypic traits.  We assume that the spatial diffusion rate is parameterized by the variable $\alpha$, while  mutation is modeled by a diffusion process with constant rate $\epsilon^2>0$. Each individual is in competition for resources with all other individuals at the same spatial location. Denoting by $u(t,x,\alpha)$ the population density of the species with trait $\alpha \in \mathcal{A}$ at location $x\in D$ and time $t>0$, the model is given as
\begin{equation}\label{eq:mainp}
\left\{
\begin{array}{ll}
u_t = \alpha {\Delta} u  + \left[m(x) - \hat u (x,t))\right] u + \epsilon^2 u_{\alpha\alpha}, &x\in D, \alpha \in (\underline\alpha, \bar\alpha), t>0,\\
\frac{\partial u}{\partial n} = 0, &x\in \partial D, \alpha \in (\underline\alpha, \bar\alpha), t>0,\\
u_\alpha = 0, &x \in D, \alpha \in \{\underline\alpha,\overline\alpha\}, t>0,\\
u(0,x,\alpha) = u_0(x,\alpha), & x\in D, \alpha \in  (\underline\alpha, \bar\alpha).
\end{array}
\right.
\end{equation}
Here $\Delta = \sum_{i=1}^N \frac{\partial^2}{\partial x_i^2}$ denotes the Laplace operator in the spatial variables, $$
\hat{u}(x,t):=\int_{\underline\alpha}^{\overline\alpha} u(t,x,\alpha)\, d\alpha,
$$
$n$ denotes the outward unit normal vector on the boundary $\partial D$ of the spatial domain $D$, and $\frac{\partial}{\partial n} = n \cdot \nabla$. The function $m(x)$ represents the quality of the habitat, which is assumed to be non-constant in $x$ to reflect that the environment is spatially heterogeneous but temporally constant.

The model \eqref{eq:mainp} can be viewed as a continuum (in trait) version of the following mutation-selection model considered by Dockery et al. \cite{DHMP}, concerning the competition of $k$ species with different dispersal rates but otherwise identical:
\begin{equation}\label{eq:DHMP2}
\left\{
\begin{array}{ll}
\frac{\partial}{\partial t} u_i = \alpha_i \Delta u_i + \left[m(x) - \sum_{j=1}^k u_j\right] u_i + \epsilon^2 \sum_{j=1}^k M_{ij} u_j  \\
\qquad \qquad\qquad \qquad\qquad \qquad\qquad \qquad \text{ in }D \times(0,\infty), i=1,...,k,\\
\frac{\partial }{\partial n}u_i = 0 \quad \qquad \qquad\qquad \qquad \qquad \quad \text{ on }\partial D \times(0,\infty), i=1,...,k,\\
u_i(x,0) = u_{i,0}(x)\qquad \qquad\quad\,\,\,\, \quad \quad \quad \text{ in } D, i = 1,...,k,
\end{array}
\right.
\end{equation}
where $0<\alpha_1 < \alpha_2 < ... < \alpha_k$ are constants, $m(x) \in C^2(\overline{D})$ is non-constant, $M_{ij}$ is an irreducible real $k\times k$ matrix that  models the mutation process so that $M_{ii} <0$ for all $i$, and $M_{ij} \geq 0$ for $i\neq j$ and $\epsilon^2\geq 0$ is the mutation rate.

Model \eqref{eq:DHMP2} was introduced to address the
%
%Let the spatial habitat be given by a bounded domain $D$ in $\mathbb{R}^N$ with smooth boundary $\partial D$, and denote by $n$ the outward unit normal vector on $\partial D$, and $\frac{\partial}{\partial n} = n \cdot \nabla$. Let $m(x)$ represent the quality of the habitat, which is assumed to be spatially heterogeneous but temporally constant.
%
 question of evolution of random dispersal. In the case when  there is no mutation, i.e. $\epsilon=0$,
 this question was
  considered in \cite{Hastings}, where it was shown that in a competition model of two species with different diffusion rates but otherwise identical, a rare competitor can invade the resident species if and only if
   the rare species is the slower diffuser.  Dockery et al. \cite{DHMP}
generalized the work of
Hastings \cite{Hastings} to $k$ species situation, and proved that no two species can coexist at equilibrium, i.e.
  the set of non-trivial, non-negative steady states of the system \eqref{eq:DHMP2} is given by
$$
\{(\theta_{\alpha_1}, 0,..., 0), (0,\theta_{\alpha_2},0,...,0), ... , (0,..., \theta_{\alpha_k})\},
$$
where $\theta_\alpha$ is the unique positive solution of
\begin{equation*}%\label{eq:thetaaa}
\alpha\Delta\theta+\theta(m-\theta)=0\quad \mbox{in}\ D,
\quad \quad \tfrac{\partial\theta}{\partial n}=0 \quad \mbox{on}\ \partial D.
\end{equation*}
Moreover,  among the non-trivial steady states, only $(\theta_{\alpha_1}, ...,0)$, the steady state where the slowest diffuser survives, is stable and the rest of the steady states are all  unstable. Furthermore, when $k=2$,  the steady state $(\theta_{\alpha_1},0)$ is globally asymptotically stable among all non-negative, non-trivial solutions. Whether such a result holds for three or more species remains an interesting and important open question.

Dockery et al. \cite{DHMP} further inquired the effect of small mutation. More precisely,  when $0<\epsilon \ll 1$,
it is shown that \eqref{eq:DHMP2} has a unique steady
state $\tilde{U}= (\tilde u_1, \tilde u_2, ..., \tilde u_N)$ in the space of non-trivial,
non-negative functions,
 such that $\tilde{u}_i >0$ for all $i$, and $\tilde{U} \to (\theta_{\alpha_1},0, ...,0)$ as $\epsilon \to 0$;
  i.e. the system \eqref{eq:DHMP2} equilibrates only when the slowest species is dominant and all other species remain at low densities. %The mutation selection model we study in the present paper can be viewed as a partial generalization of the results of \cite{DHMP} to infinite traits.

%
%\begin{equation}\label{eq:DHMP}
%\left\{
%\begin{array}{ll}
%u_t = \alpha_1 {\Delta} u + (m(x) - u - v)u &\text{ in }D \times(0,\infty),\\
%v_t = \alpha_2 {\Delta} v + (m (x) - u -v)v &\text{ in }D \times(0,\infty),\\
%\frac{\partial u}{\partial n} = \frac{\partial v}{\partial n} = 0 &\text{ on }\partial D \times(0,\infty),\\
%u(x,0) = u_0(x),\quad v(x,0) = v_0(x)&\text{ in }D,
%\end{array}
%\right.
%\end{equation}
%where it was proven that if $\alpha_1 < \alpha_2$, then the steady state $(\theta_{\alpha_1},0)$ is globally asymptotically stable among all non-negative, non-trivial initial data. Here $\theta_{\alpha_1}$ is the unique positive solution to
%$$
%%\left\{
%\alpha_1 \Delta\theta + (m - \theta)\theta=0\quad \text{ in }D,\quad \text{ and } \quad \frac{\partial \theta}{\partial n}=0 \quad \text{ on }\partial D,
%%\right.
%$$
%i.e. the species with the smaller diffusion rate not only invades but also wipes out the faster competitor, regardless of the initial data.

It is natural, then, to inquire if the situation in the discrete (in trait) framework carries over to the continuum framework.
The aim of this paper is to study the asymptotic behavior of steady state(s) of \eqref{eq:mainp}.
%In particular, we shall  provide a partial answer to this question. %inquire how the mutation-selection model \eqref{eq:mainp} is related to the discrete (in trait) model \eqref{eq:DHMP2}, especially for small values of $\epsilon$.
%
%This paper aims to provide a partial answer to this question.
Let $u_\epsilon$ be any positive steady state of \eqref{eq:mainp}, we will show that, as $\epsilon \to 0$,
$$
u_\epsilon(x,\alpha) \to \delta(\alpha - \underline\alpha) \theta_{\underline\alpha}(x),
$$
i.e. $u_\epsilon$ converges to a Dirac mass supported at the lowest possible trait value $\underline\alpha$.
See Theorem \ref{thm} for precise descriptions of our main results.

%(Possibly omit this paragraph)
%We may interpret the above result in the adaptive dynamics framework \cite{AD}, which assumes ecological interactions to operate in a much faster time scale than mutation, so that one may consider the fate of a resident species at ecological equilibrium $(\theta_{\alpha_1},0)$ facing possible invasion by a rare mutant invader. In this framework, the results of Hastings and Dockery et al. imply that (i) if $0<\alpha_2 < \alpha_1$, then the invader $v$ invades and drives out the resident $u$; (ii) if $\alpha_2 > \alpha_1$, then the invader $v$ fails to invade and goes to extinction.  In particular, if one limits the random diffusion rates $\alpha_1, \alpha_2$ to lie within certain physical range $\mathcal{A}=[\underline\alpha,\overline\alpha]$, then a resident adopting diffusion rate $\underline\alpha$ cannot be invaded by any mutant $v$ with a different diffusion rate; i.e. $\underline\alpha$ is the unique Evolutionarily Stable Strategy (ESS) \cite{MP} among diffusion rates in $\mathcal{A}$. We note here that if we (formally) consider the case $\epsilon =0$ then \eqref{eq:mainp} defines a dynamical systems in measures, and \eqref{eq:DHMP2} can be identified as a subsystem of \eqref{eq:mainp} (but with $\epsilon=0$)  by considering initial conditions of the form $u_0(x,\alpha) = \sum_{j=1}^N\delta(\alpha - \alpha_j) u_{j,0}(x).$

Mutation-selection models for a continuum of trait values have  been studied extensively, when the phenotypic trait is associated only with growth advantages \cite{BRV, burger,calsina,DJMP,LMP, MW,PB}. See also \cite{JR} for a pure selection model. The consideration of a spatial trait is more recent \cite{ACR,ADLP, BM,  MiP}.

%- discrete approach, separation of scales - adaptive dynamics - hastings - dockery.

System \eqref{eq:mainp} is also considered in an unbounded spatial domain $x \in \mathbb{R}$. A formal argument concerning the existence of an ``accelerating wave" is presented in \cite{BCMMPRV}, which provides a theoretical explanation of the accelerating invasion front of cane toads in Australia \cite{toadbio}. Rigorous results are obtained when $\alpha \in \mathcal{A} = [\underline\alpha , \overline\alpha]$ more recently in \cite{BC, Turanova2014}. It can be summarized that the highest diffusion rate is selected when the underlying spatial domain is unbounded, which stands in contrast to the case of bounded spatial domains we consider in this paper, where the lowest possible diffusion rate is selected.

The rest of the paper is organized as follows: The main results are stated in Section \ref{sec:2}. Section \ref{sec:3} concerns various estimates on steady states of \eqref{eq:mainp}. In Section \ref{sec:4} we introduce an auxiliary eigenvalue problem and  a transformed problem of \eqref{eq:main}. The limit of $\hat{u}_\epsilon$ is determined in Section \ref{subsect:hatu}.  In Section \ref{sec:6}, we analyze the qualitative properties of solutions  to the transformed problem. The proof of our main result is given in Section \ref{sec:7}. Finally, the Appendices \ref{sec:A} to \ref{sec:C} establish the existence results, the smooth dependence of principal eigenvalue on coefficients as well as a Liouville-type results concerning positive harmonic functions on cylinder domains.

%Finally, a Liouville-type result is proved in the Appendix.

\section{Main Results}\label{sec:2}

%Let $u_\epsilon$ denote a positive steady state of \eqref{eq:main}

In this paper, we consider the asymptotic behavior of positive steady states of \eqref{eq:mainp}, denoted by $u_\epsilon$. That is, $u_\epsilon$ satisfies the following  mutation-selection equation of a randomly diffusing population:
\begin{equation}\label{eq:main}
\left\{
\begin{array}{ll}
\alpha {\Delta} u_\epsilon + \epsilon^2 (u_\epsilon)_{\alpha\alpha} + \left[m(x) - \hat u_\epsilon (x)\right] u_\epsilon = 0&\text{ in }\Omega:= D \times (\underline\alpha, \bar\alpha),\\
\frac{\partial u_\epsilon}{\partial n} = 0 &\text{ on }\partial D \times (\underline\alpha , \overline\alpha),\\
(u_\epsilon)_\alpha = 0 &\text{ in }D \times \{\underline\alpha,\overline\alpha\},
\end{array}
\right.
\end{equation}
where
\begin{equation}\label{eq:u-epsilon-a}
\hat{u}_\epsilon(x) = \int_{\underline\alpha}^{\overline\alpha} u_\epsilon(x,\alpha)\,d\alpha.
\end{equation}

Throughout this paper, we assume
%\begin{description}
$$
\hspace{-1cm}\textup{(A)} \quad  m(x)\text{ is a non-constant function in }C({\bar D})\text{ such that }\int_D m(x)\, dx >0.
$$
%\end{description}

The existence of positive solutions to \eqref{eq:main} can be stated
as follows:
%Our main interest in this paper is the asymptotic behaviors of these positive solutions
%as $\epsilon\to 0$.
%Our first result is the following characterization of the existence of positive solution to \eqref{eq:main}.
\begin{theorem}\label{thm:existence}
Suppose {\rm (A)} holds, then \eqref{eq:main} has at least one positive solution for all $\epsilon>0$.
\end{theorem}
We postpone the proof of Theorem \ref{thm:existence} to Appendix \ref{sec:A}. For the rest of the paper we will focus on the asymptotic behavior of
positive solutions of \eqref{eq:main}  as $\epsilon \to 0.$
To this end, we define the following quantities:

\begin{definition}\label{def:2.2}
\begin{itemize}

\item[{\rm (i)}] Let $\theta_{\underline\alpha}(x)$ be the unique positive solution of
\begin{equation}\label{eq:theta}
\left\{
\begin{array}{ll}
\underline\alpha {\Delta} \theta + \theta(m(x) - \theta) = 0&\text{ in }D,\\
\frac{\partial\theta}{\partial n} =0 &\text{ on }\partial D.
\end{array}
\right.
\end{equation}

\item[{\rm (ii)}]
For each  $\alpha \in [\underline\alpha,\overline\alpha]$, we denote the principal eigenvalue and principal positive eigenfunction of the following problem  by $\sigma^*(\alpha)$ and $\psi^*(x,\alpha)$, respectively:
\begin{equation}\label{eq:psistar}
\left\{
\begin{array}{ll}
\alpha {\Delta} \psi + (m(x) - \theta_{\underline\alpha}(x))\psi +\sigma\psi = 0\quad \text{ in }D,\\
\frac{\partial\psi}{\partial n} =0 \quad \text{ on }\partial D, \quad \text{ and }\quad \int_D \psi^2\,dx =
\int_D \theta^2_{\underline\alpha}\,dx.
\end{array}
\right.
\end{equation}
(Note that by (i),  $\theta_{\underline\alpha}(x)$ is a positive eigenfunction for \eqref{eq:psistar} when $\alpha = \underline\alpha$. By uniqueness of the (normalized) principal eigenfunction, we have $\sigma^*(\underline\al) = 0$, and  $\psi^*(x,\underline\alpha) = \theta_{\underline\alpha}(x)$ for $x \in D$.)

\item[{\rm (iii)}] Denote by $\eta^*(s)$ the unique positive solution to
\begin{equation}\label{eq:etastar}
\left\{
\begin{array}{ll}
\eta'' + (a_0 -  a_1 s)\eta = 0\quad \text{ for }s>0,\\
\eta'(0) = 0 = \eta(+\infty) \quad \,\,\,\text{ and } \int_0^\infty \eta(s)\,ds = 1,
%\left.\frac{\partial \sigma^*}{\partial \alpha}\right|_{\alpha= \underline\alpha}.
\end{array}
\right.
\end{equation}
where $a_0, a_1$ are positive constants determined by $a_1 = \frac{\partial \sigma^*}{\partial \alpha}(\underline\alpha)$ and $a_0 = (a_1)^{2/3} A_0$, where $A_0$ is  the absolute value of the first negative zero of the derivative of the Airy function.

\end{itemize}
\end{definition}

When $m(x) \equiv 1$, one can easily show that $u_\epsilon \equiv 1/(\overline\alpha - \underline\alpha)$, i.e. there is no selection in the trait variable. Our main result shows that the outcome changes drastically when $m(x)$ is non-constant. In fact, $u_\epsilon$ concentrates at the lowest value in the trait variable, as $\epsilon \to 0$. This phenomenon is also known as spatial sorting.

%The situation when $m(x)$ is non-constant is completely different.

%The concentration in the lowest diffusion trait $\underline\alpha$ when $m(x)$ is non-constant should be compared with the case when $m(x) \equiv 1$, whereby $u_\epsilon \equiv 1/(\overline\alpha - \underline\alpha)$ and there is no selection in the trait variable.

%Our main result concerns the asymptotic behavior of solution $u_\epsilon$ as $\epsilon \to 0$.
\begin{theorem}\label{thm}
Let  $u_\epsilon$ be any positive solution of \eqref{eq:main}. Then for all $\beta>0$,  there exists $C>0$ independent of $\epsilon>0$ such that
\begin{equation}\label{eq:thm1}
 u_\epsilon(x,\alpha) \leq C\epsilon^{-2/3}\exp\left( - \beta(\alpha - \underline\alpha)\epsilon^{-2/3} \right)
\end{equation}
in $\Omega = D \times(\underline\alpha,\overline\alpha)$. Moreover,  as $\epsilon \to 0$
\begin{equation}\label{eq:thm2}
\left\|
\epsilon^{2/3}u_\epsilon(x,\alpha) - \theta_{\underline\alpha}(x)\eta^*\left( \frac{\alpha - \underline\alpha}{\epsilon^{2/3}}\right)\right\|_{L^\infty(\Omega)} \to 0 %\quad \text{ as }\epsilon \to 0,
\end{equation}
where $\theta_{\underline\alpha}(x)$
and $\eta^*(s)$ are given as above.
In particular, we have
\begin{equation}\label{eq:thm2-aa}
\hat u_\epsilon(x) = \int_{\underline\alpha}^{\overline\alpha} u_\epsilon(x,\alpha)\,d\alpha \to \theta_{\underline\alpha}(x)\quad \text{ as }\epsilon \to 0.
\end{equation}
\end{theorem}

%\section{Proof of Theorem \ref{thm:existence}}
%
%As in Hao-Lam-Lou draft.

%\begin{remark}

%\end{remark}

As the proof of Theorem \ref{thm} is fairly technical, we briefly outline the main ingredients
for readers. Our idea is to establish the ``separation of variables"
formula \eqref{eq:thm2} for $u_\epsilon$, by
introducing the scaling $s=(\alpha-\bar\alpha)/\epsilon^{2/3}$ 
 and writing  $u_\epsilon(x, \alpha)=\psi_\epsilon (x, \alpha) w_\epsilon(x, s)$,
where $\psi_\epsilon$ is the principal eigenfunction of 
$-\alpha\Delta \psi+(\hat{u}_\epsilon-m)\psi=\sigma\psi$, 
subject to the zero Neumann boundary condition and the integral constraint 
$\int_D \psi^2=\int_D \theta_{\underline\alpha}^2$. 
The main body of our paper is devoted to the proof of following two things: 
(i) As $\epsilon \to 0$, $\hat{u}_\epsilon\to \theta_{\underline\alpha}$ uniformly,
which implies that $\psi_\epsilon(x, \underline\alpha+\epsilon^{2/3}s)
\to \theta_{\underline\alpha}$. The concentration phenomenon of $u_\ep$
on the subset $D \times \{\underline\al\}$ of $\Omega$, i.e. \eqref{eq:thm2-aa},
is established in Section 5 with the help of several key estimates
proved in Sections 3 and 4, 
 such as the $L^\infty$ estimate of $\hat{u}_\ep$, as well as 
 the limit $\lim_{\ep \to 0} (\hat{u}_\ep- m)$ being non-constant;
(ii)  As $\epsilon \to 0$, $\epsilon^{-2/3}w_\epsilon(x, s)$ converges to $\eta(s)$
uniformly for some function $\eta$ which is independent of the spatial variable $x$. 
This is done in Section 6 along with some key estimates established in the earlier sections.

% outline the rest of the paper as follows. In Section \ref{sec:3}, the $L^\infty$ estimate of $\hat{u}_\ep(x)$, as well as the key result that the limit $\lim_{\ep \to 0} (\hat{u}_\ep- m(x))$ being non-constant, are proved. In Section \ref{sec:4}, we collect some results regarding two auxiliary eigenvalue problems, and introduce the scaling $s = \ep^{-2/3}(\al - \underline\al)$ which regularizes the problem. Section \ref{subsect:hatu} is devoted to establish the concentration phenomenon of $u_\ep$ on the subset $D \times \{\underline\al\}$ of $\Omega$, which implies $\hat{u}_\ep \to \theta_{\underline\al}$. In Sections \ref{sec:6} and \ref{sec:7}, we determine the limiting profile of $u_\ep$, under the new coordinates $(x,s)$, and prove Theorem \ref{thm}. In Appendices \ref{sec:A} and \ref{sec:B}, we present the proof of existence of $u_\ep$ and the proof of  smooth parameter dependence  of eliptic eigenvalue problems respectively. Finally,   a Liouville theorem in cylinder domains, which is the main technical tool in this work, is stated and proved in Appendix \ref{sec:C}.

\begin{remark}
{\rm
After this work is completed, the authors learned that a closely related result, under a slightly different formulation, is independently proved by B. Perthame and P.E. Souganidis
%in {\rm \cite{PS}}
under a different approach,
%[Perthame-Souganidis, arXiv:1505.03420],
where an intermediate trait attains the minimum diffusion rate and an interior Dirac mass  is found when
the mutation rate tends to zero. Apart from the distinction in our approaches, we note the following distinct features of our work: (i) A boundary concentration is found in our set-up, instead of an interior concentration in {\rm \cite{PS}} which predicts different scalings in powers of $\ep$; (ii)  Our method does not assume the convexity of spatial domain $D$;   (iii) Various detailed $L^\infty$ estimates and asymptotic limits are obtained (Theorem \ref{thm}) which paves the way to the proof of asymptotic stability and uniqueness of $u_\ep$ in a future paper; (iv) The key estimate of the limit $h_0(x)= \lim_{\ep \to 0} \hat{u}_\ep(x) - m(x)$ being non-constant (Lemma \ref{lem:2}) reflects the effect of spatial heterogeneity, the
 underlying mathematical reason for the selection of small diffusion rate. See also Proposition \ref{prop:PS} which makes the connection to {\rm  \cite[Lemma 4.3]{PS}}.
 }
\end{remark}

\section{Properties of $\hat u_\epsilon$}\label{sec:3}

In this section we establish various properties of $\hat u_\epsilon$.
Recall that $\hat u_\epsilon$ is defined in \eqref{eq:u-epsilon-a}.
\begin{lemma}\label{lem:1}
There exists some positive constant
$\delta_1=\delta_1(\underline\alpha, \bar\alpha, m)$ independent of $\epsilon$ such that
$$
\delta_1 \leq \hat u_\epsilon(x) \leq 1/\delta_1 \quad \text{ in }D
$$
for all $\epsilon>0$. In particular,
\begin{equation}\label{eq:h}
h_\epsilon(x):= \hat u_\epsilon(x) - m(x)
% \rightharpoonup h_0(x)
\end{equation}
is bounded uniformly in $L^\infty(D)$.
%and we may pass to a subsequence $\epsilon_k \to 0$ such that for some $h_0 \in L^\infty(D),$
%weakly in $L^p(D)$ for all $p>1$.
\end{lemma}
\begin{proof}
The idea of the upper bound follows from  \cite{Turanova2014}.
Define
\begin{equation}\label{eq:v0}
v_\epsilon(x) = \int_{\underline\alpha}^{\bar\alpha} \alpha u_\epsilon(x,\alpha)\,d\alpha.\end{equation} Then we have
\begin{equation}\label{eq:compare}
\underline\alpha \hat u_\epsilon(x) \leq v_\epsilon(x) \leq \bar\alpha \hat u_\epsilon(x) \quad \text{ in }D.
\end{equation}
Integrating \eqref{eq:main} over $\alpha$ gives
\begin{equation}\label{eq:v}
\left\{
\begin{array}{ll}
{\Delta} v_\epsilon(x) + (m(x) - \hat u_\epsilon(x))\hat u_\epsilon(x) = 0&\text{ in }D,\\
\frac{\partial v_\epsilon}{\partial n} = 0&\text{ on }\partial D.
\end{array}
\right.
\end{equation}
Let $\max_{{\bar D}} v_\epsilon = v_\epsilon(x_0)$, then $\hat u_\epsilon(x_0) \leq m(x_0) \leq \max_{{\bar D}} m$ (see \cite[Proposition 2.2]{LN}), and by \eqref{eq:compare},
$$
\underline\alpha \max_{{\bar D}} \hat u_\epsilon \leq \max_{{\bar D}} v_\epsilon =  v_\epsilon(x_0) \leq {\bar\alpha} \hat u_\epsilon(x_0) \leq  {\bar\alpha} \max_{{\bar D}} m.
$$
Hence
%\begin{equation}\label{eq:h}
%h_\epsilon(x) := \hat u_\epsilon(x) - m(x),
%\end{equation}
we deduce that
 $\|\hat u_\epsilon\|_{L^\infty(D)}$ and $\|h_\epsilon\|_{L^\infty(D)}$ are bounded uniformly in  $\epsilon$, where $h_\epsilon(x)= \hat u_\epsilon(x) - m(x)$ is given in \eqref{eq:h}.

Next, we show the lower bound of $\hat u_\epsilon$.
By \eqref{eq:compare}, we deduce that $$\hat{u}_\epsilon(x) = k_\epsilon(x) v_\epsilon(x)$$ for some  $k_\epsilon(x) \in L^\infty(D)$ such that ${\bar\alpha}^{-1} \leq k_\epsilon(x) \leq {\underline\alpha}^{-1}$. So that $v_\epsilon$ is a positive solution of
$$
-{\Delta} v_\epsilon + h_\epsilon(x) k_\epsilon(x) v_\epsilon = 0 \quad \text{ in }D,\quad \text{ and }\quad \frac{\partial v_\epsilon}{\partial n}=0\quad \text{ on }\partial D,
$$
where we have already shown that $h_\epsilon = \hat{u}_\epsilon - m$ is uniformly bounded (in $L^\infty(D)$) in $\epsilon$. Therefore, the Harnack inequality applies so that
\begin{equation}\label{eq:harnack1}
\max_{{\bar D}} v_\epsilon \leq C' \min_{{\bar D}} v_\epsilon
\end{equation}
for some constant $C'>1$ independent of $\epsilon$. Combining with \eqref{eq:compare}, we have
\begin{equation}\label{eq:harnack}
\underline\alpha\max_{{\bar D}}\hat{u}_\epsilon
\leq \max_{{\bar D}}v_\epsilon
\leq {C'} \min_{{\bar D}} v_\epsilon
\leq {C'\bar\alpha}\min_{{\bar D}} \hat u_\epsilon.
\end{equation}

%Suppose to the contrary that $\min_{{\bar D}} \hat u_\epsilon \to 0$, then by   \eqref{eq:harnack}, $u_\epsilon\to 0$ uniformly in $D$.
Now,
if we divide \eqref{eq:main} by $u_\epsilon$ and integrate by parts over $\Omega = D \times(\underline\alpha, \overline\alpha)$, we obtain
\begin{equation}\label{eq:byparts}
(\overline\alpha - \underline\alpha) \int_D (\hat u_\epsilon - m)\,dx = \int_\Omega h_\epsilon(x)\,d\alpha dx = %\frac{1}{\bar\alpha - \underline\alpha}
\int_\Omega \frac{ \alpha |\nabla_x u_\epsilon|^2 + \epsilon^2 |(u_\epsilon)_\alpha|^2}{u_\epsilon^2} >0.
\end{equation}
We deduce  by \eqref{eq:harnack} and \eqref{eq:byparts} that
$$
\frac{C' \overline\alpha}{\underline\alpha}\min_{\overline D} \hat{u}_\epsilon \geq \max_{\overline D} \hat{u}_\epsilon \geq \frac{1}{|D|} \int_D \hat{u}_\epsilon(x)\,dx \geq \frac{1}{|D|} \int_D m(x)\,dx >0.
$$
This establishes the uniform lower bound of $\hat u_\epsilon$.
%Taking $\int_D \hat u_\epsilon\,dx \to 0$, we deduce that $\int_D m(x)\,dx \leq 0$, this is a contradiction to (A).
%
%Letting $\epsilon \to 0$ in \eqref{eq:byparts}, we deduce that $\int_D m(x)\,dx \leq 0$.
%This contradiction to assumption (A) proves the uniform lower bound of $\hat u_\epsilon$.
\end{proof}

\begin{remark}\label{rmk:0}
{\rm Since $\|\hat{u}_\epsilon\|_{L^\infty(D)}$ and $\|v_\ep\|_{L^\infty(D)}$ are bounded uniformly in $\epsilon$, applying elliptic $L^p$ estimate to \eqref{eq:v} implies that $\|v_\epsilon\|_{W^{2,p}(D)}$ is bounded uniformly in $\epsilon$. In particular, there exists sequence $\epsilon_k \to 0$ such that $v_{\epsilon_k}$ converges uniformly on $D$.}
\end{remark}

%\begin{remark}
%If instead $(A')$ is assumed instead of $(A)$, then the proof for the upper bound is valid, but the proof for the lower bound fails.
%\end{remark}

\begin{lemma}\label{lem:uinfty}
There exists a constant $C>0$ such that for any positive solution $u_\epsilon$ of \eqref{eq:main},
$$
\sup_{ D \times (\underline\alpha,\overline\alpha)} u_\epsilon \leq C \epsilon^{-1}.
$$
\end{lemma}

\begin{proof}
Choose $x_\epsilon$ and $\alpha_\epsilon$ such that the supremum of $u_\epsilon$ is attained at $(x_\epsilon, \alpha_\epsilon)$.
%Passing to a subsequence, one may assume that $\alpha_\epsilon \to \alpha_0$
%as $\epsilon\to 0$.
Next, let $U_\epsilon(x,\tau) = u_\epsilon(x, \alpha_\epsilon + \epsilon\tau)$,
%/ u_\epsilon(x_\epsilon, \alpha_\epsilon)$,
then
 (extending $u_\epsilon$ to $D\times[\underline \al - \ep_0, \overline\al + \ep_0]$  by reflection across the boundary portions $D \times\{ \underline\alpha, \overline\alpha\}$ if necessary)
  one
may observe that $U_\epsilon$ satisfies a uniformly elliptic equation with uniformly bounded (in $L^\infty)$ coefficients
$$
\left\{
\begin{array}{ll}
\alpha \Delta_x U_\epsilon + U_{\epsilon,\tau\tau} + h_\epsilon(x) U_\epsilon =0 & \text{ in }D \times [-2,2],
\\
\frac{\partial U_\epsilon}{\partial n} =0 & \text{ on }\partial D \times [-2,2],
\end{array}\right.
$$
where $\alpha = \alpha_\epsilon + \ep \tau$ is always bounded between $[\underline\al - \ep_0, \overline\al, + \ep_0] \subset (0,+\infty)$. Hence, we may apply the Harnack's inequality to yield a positive constant $C$ independent of $\epsilon$ such that
$$
u_\epsilon(x,\alpha_\epsilon + \epsilon \tau) \geq C u_\epsilon(x_\epsilon,\alpha_\epsilon) = C\sup_{D\times (\underline\alpha, \overline\alpha)} u_\epsilon
$$
for all $x \in D$, $\tau \in [-1,1]$.
%
%  can show via $L^p$ estimate that $U_\epsilon$ converges weakly in $W^{2,p}_{loc}(D \times \mathbb{R})$ (and hence   locally uniformly) to a  non-negative  function $U$ satisfying
%$$
%\alpha_0 \Delta_x U + U_{\tau\tau} + h_0(x) U =0 \text{ on } D\times \mathbb{R},
%$$
%where $h_0$ is a subsequential limit of $h_\epsilon = \hat{u} - m$ in the weak $L^p(D)$ sense.
% By Proposition  \ref{prop:A1}, and the fact that $U_\epsilon(x_\epsilon,0) =1$,
%$U \equiv 1$. Hence
Hence,
$$
\hat{u}_\epsilon(x) = \int_{\underline\alpha}^{\overline\alpha} u_\epsilon(x,\alpha)\,d\alpha \geq
% C{\epsilon} u_\epsilon(x_\epsilon, \alpha_\epsilon)=
  C{\epsilon} \sup_{D \times(\underline\alpha, \overline\alpha)} u_\epsilon
$$
for all $\epsilon$ sufficiently small. By Lemma \ref{lem:1}, we deduce that
$$
\sup_{D \times(\underline\alpha, \overline\alpha)} u_\epsilon \leq C'\epsilon^{-1}
$$
for some positive constant $C'$.
\end{proof}
%%%%%%%%%%%%%%%%%%%%%%%%%%%%%%%%%%%%%%%%%%%%%%%%%%%%%
%%%%%%%%%%%%%%%%%%%%%%%%%%%%%%%%%%%%%%%%%%%%%%%%%%%%%
%%%%%%%%%%%%%%%%%%%%%%%%%%%%%%%%%%%%%%%%%%%%%%%%%%%%%
%%%%%%%%%%%%%%%%%%%%%%%%%%%%%%%%%%%%%%%%%%%%%%%%%%%%%
By Lemma \ref{lem:1}, $h_\epsilon$ is bounded in $L^\infty(D)$ uniformly in $\epsilon$. Therefore, up to subsequences $\epsilon_j \to 0$, $h_{\epsilon_j}$ converges weakly in $L^p(D)$ for all $p>1$.
We first prove an important property of any subsequential limit $h_0$.

\begin{lemma}\label{lem:2}
Let $h_0$ be a weak (subsequential)  limit of $h_\epsilon(x)$ in $L^p(D)$ ($p>1$) as $\epsilon \to 0$, then $h_0(x)$ is non-constant in  $D$.
\end{lemma}

\begin{proof}
Suppose to the contrary that for some $c \in \mathbb{R}$, $h_\epsilon(x) \rightharpoonup c $ weakly in $L^p(D)$ for all $p>1$.

\begin{claim}\label{claim:c=0}
 $c=0$.
 \end{claim}

By taking $\epsilon \to 0$ in \eqref{eq:byparts}, we deduce that $c \geq 0;$
%
% This can be observed by dividing \eqref{eq:main} by $u_\epsilon$ and integrate over $\Omega = D \times(\underline\alpha, \overline\alpha)$, so that
%$$
%\int_\Omega h_\epsilon\,dx = \frac{1}{\bar\alpha - \underline\alpha} \int_\Omega \frac{ \alpha |\nabla_x u_\epsilon|^2 + \epsilon^2 |(u_\epsilon)_\alpha|^2}{u_\epsilon^2} >0.
%$$
%
i.e. for some $c\geq 0$, $\hat u_\epsilon = h_\epsilon + m(x) \rightharpoonup c + m(x)$ in $L^p$ for all $p>1$.

Next, integrating \eqref{eq:main} with respect to $\alpha$ and then $x$, we obtain
%\begin{equation}\label{eq:identity}
\begin{equation}\label{eq:lem1-1}
\int_D \hat u_\epsilon(x)( m(x) - \hat u_\epsilon(x))=0,
\end{equation}
%\end{equation}
so that
\begin{equation}\label{eq:c}
\int_D m(x) (m(x) + c) = \lim_{\epsilon \to 0} \int_D m(x) \hat u_\epsilon \geq \liminf_{\epsilon \to 0} \int_D (\hat u_\epsilon)^2 \geq \int_D (m(x) + c)^2,
\end{equation}
where the first inequality follows from \eqref{eq:lem1-1}, and the second inequality  from expanding $\int_D[\hat u_\epsilon(x) - (m(x) + c)]^2 \geq 0$ as
$$
\int_D \hat{u_\epsilon}^2 \geq 2\int_D \hat u_\epsilon (m + c) - \int_D (m+c)^2 \to \int_D (m+c)^2.
$$
Hence, by \eqref{eq:c}, we have $c=0$; i.e. $h_\epsilon(x) \rightharpoonup 0$ weakly in $L^p(D)$ for all $p>1$.

%\begin{remark}

%\end{remark}

Note that we are done if $m<0$ somewhere, since then $h_\epsilon(x) \geq -m(x) >0$ in some open subset of $D$ independent of $\epsilon$,  which contradicts $h_\epsilon \rightharpoonup 0$. For the general case of  $m(x)$ being possibly non-negative, we continue via a blow-up argument.
Let
$$
C_\epsilon:= \sup_{\alpha \in (\underline\alpha, \overline\alpha)}
\left( \frac{\sup_{x \in D} u_\epsilon(x,\alpha)}{\inf_{x \in D} u_\epsilon(x,\alpha)} \right).
$$
It is enough to show that
\begin{claim}\label{claim:1}
$C_\epsilon \searrow 1$ as $\epsilon \to 0$.
\end{claim}
Assuming Claim \ref{claim:1}, then  by definition of $C_\epsilon$, %there exists a $ $\alpha$ for all $\alpha$,
$$
 u_\epsilon (x, \alpha) \leq C_\epsilon  u_\epsilon(y,\alpha) \quad \text{ for all }x,y \in D \text{ and }\underline\alpha < \alpha < \overline\alpha.
$$
This gives, upon integrating over $\alpha \in (\underline\alpha, \overline\alpha)$,
$$
\sup_{x \in D} \hat u_\epsilon(x) \leq C_\epsilon \inf_{y \in D} \hat u_\epsilon(y).
$$
Hence $\hat u_\epsilon(x)$ converges to a constant. But this also means that $h_\epsilon = \hat u_\epsilon - m(x)$ converges to a non-constant function, as $m(x)$ is non-constant. This is a contradiction.

It remains to prove Claim \ref{claim:1}. Assume to the contrary that there exist some constant $c_0>1$, and sequences $\epsilon_k \to 0$, $\alpha_k \to \alpha_0$, $x_k, y_k \in D$ such that
\begin{equation}\label{eq:c2}
u_{\epsilon_k}(x_k,\alpha_k) \geq c_0 u_{\epsilon_k}(y_k, \alpha_k).
\end{equation}
Extend $u_\epsilon$ to $D \times [\underline\alpha - \epsilon_0, \overline\alpha + \epsilon_0]$ for some fixed $\epsilon_0$ small by reflection on the boundary $D \times \{\underline\alpha, \overline\alpha\}$, and define
$$
U_k(x,s):= \frac{u_{\epsilon_k} (x,\alpha_k + \epsilon_k s/\sqrt{\alpha_k})}{\sup_{x\in D} u_{\epsilon_k}(x,\alpha_k)} \quad \text{ in }D \times \left( \frac{\underline\alpha - \epsilon_0 - \alpha_k}{\epsilon_k}, \frac{\overline\alpha + \epsilon_0 - \alpha_k}{\epsilon_k}\right).
$$
Then \eqref{eq:c2} says that for some $c_0>1$ independent of $k$
\begin{equation}\label{eq:c222}
\inf_{x \in D} U_k(x,0) \leq \frac{1}{c_0}\quad \text{ for all }k.
\end{equation}  Moreover, $U_k$ satisfies
$$
\left\{
\begin{array}{ll}
\Delta_x U_k + \frac{\alpha_k}{\alpha_k + \epsilon_k s/\sqrt{\alpha_k}} U_{k,ss} + \frac{h_\epsilon(x)}{\alpha_k + \epsilon_k s/\sqrt{\alpha_k}} U_k =0 \quad \text{ in }D \times \left( \frac{\underline\alpha - \epsilon_0 - \alpha_k}{\epsilon_k}, \frac{\overline\alpha + \epsilon_0 - \alpha_k}{\epsilon_k}\right),\\
U_k(x,s) >0 \quad\text{ in }D \times \left(  \frac{\underline\alpha - \epsilon_0 - \alpha_k}{\epsilon_k}, \frac{\overline\alpha + \epsilon_0 - \alpha_k}{\epsilon_k}\right),\quad  \sup_{x \in D} U_k(x,0) = 1.
\end{array}
\right.
$$
Since $\alpha_k \to \alpha_0  \in [\underline\alpha, \overline\alpha]$, $\epsilon_k \to 0$, the domain of $U_k$ converges to $D \times \mathbb{R}$ as $k \to \infty$. Moveover, by the uniform boundedness of $\|h_{\epsilon_k}\|_{L^\infty(D)}$ in $k$ (Lemma \ref{lem:1}), we have for each $M>0$
the coefficients of the equation of $U_k(x,s)$ are bounded in $L^\infty(\bar D \times [-M,M])$ uniformly in $k$. Since $\sup_{x \in D}U_k(x,0) = 1$, together with \eqref{eq:c222}  we may apply Harnack
 inequality to obtain a constant $C= C(M)$ independent of $k$ such that
$$
C^{-1}  \leq U_k(x,s) \leq C \quad \text{ for } x \in D \text{ and }|s| \leq M.
$$
By $L^p$ estimates (applied to $\overline D \times [-M,M]$ for each $M$), there is a subsequence $U_{k_i}$ that converges uniformly in compact subsets of $\bar D \times \mathbb{R}$ to a positive solution of
${\Delta_x} U_0 + (U_0)_{ss} = 0$ on $D \times \mathbb{R}$. (The limiting domain is $D \times \mathbb{R}$ as $\frac{\underline\alpha - \epsilon_0 - \alpha_k}{\epsilon} \to -\infty$ and $\frac{\overline\alpha + \epsilon_0 - \alpha_k}{\epsilon} \to \infty$.)
Now, we apply Proposition \ref{prop:A1}  for positive harmonic functions on a cylinder domain,
%\begin{proposition}\label{prop:liouville}
%Any positive harmonic function $U_0$ on $D\times\mathbb{R}$ is necessarily a constant.
%\end{proposition}
%\begin{proof}
%Postponed to the appendix.
%\end{proof}
so that  $U_0 \equiv c_1$ for some positive constant $c_1$. Since $\sup_{x\in D} U_k(x,0) = 1$ for all $k$, we have $c_1 = 1$. In particular, we set $s =0$ and find a subsequence $U_{k_i}(x,0)$ converges  to $1$ uniformly for $x \in D$.
 This is in contradiction to \eqref{eq:c222} and proves Claim \ref{claim:1}. This completes the proof.
\end{proof}

The following result generalizes a key estimate of  \cite{PS}, proved wherein via Bernstein's method under the additional assumption that $D$ is convex. Although not needed for the rest of the paper, Proposition \ref{prop:PS} enables one to follow the elegant Hamilton-Jacobi approach as in \cite{PS} to show the concentration phenomenon.
\begin{proposition}\label{prop:PS}
Let $u_\ep$ be a positive solution of \eqref{eq:main}. Then there exists a constant $C>0$ independent of $\ep$ such that
$$
\epsilon\left\| \frac{u_{\epsilon,\alpha}}{u_\epsilon}\right\|_{L^\infty(\Omega)} + \left\| \frac{\nabla_x u_{\epsilon}}{u_\epsilon}\right\|_{L^\infty(\Omega)} \leq C.
$$
\end{proposition}
\begin{proof}
Extend the definition of $u_\ep$ to
%an open set $\Omega'$ containing
${D} \times (2\underline\alpha - \overline\alpha, 2\overline\alpha- \underline\alpha)$ by reflecting along the boundary portions $D \times \{\underline\alpha, \overline\alpha\}$. %, and then on $\partial D \times (2\underline\alpha - \overline\alpha, 2\overline\alpha- \underline\alpha)$. $\Omega'$ depends only on $\partial D$ and is independent of $\ep$.
For each $\al_0 \in [\underline\alpha, \overline\alpha]$, define
$$
U_\ep(x,\tau):= {u_\ep(x,\al_0 + \epsilon\tau)}.%{\sup_{x \in D} u_\ep(x,\alpha_0)}.
$$
Then $U_\ep$ is a positive solution to
$$
%\left\{
%\begin{array}{ll}
A(\tau,\ep)\Delta_x U_\ep + U_{\ep,\tau\tau} - h_\ep(x) U_\ep = 0 %\quad\text{ in }D \times  \left( \ep^{-1}(2\underline\alpha - \overline\alpha - \al_0), \ep^{-1}(2\overline\alpha- \underline\alpha - \al_0)\right),\\
%\frac{\partial U_\ep}{\partial n} = 0 \quad \text{ on }\partial D \times  \left( \ep^{-1}(2\underline\alpha - \overline\alpha - \al_0), \ep^{-1}(2\overline\alpha- \underline\alpha - \al_0)\right).
%\end{array}\right.
$$
in $D \times  \left( \ep^{-1}(2\underline\alpha - \overline\alpha - \al_0), \ep^{-1}(2\overline\alpha- \underline\alpha - \al_0)\right)$ and satisfies the Neumann boundary condition on $\partial D \times  \left( \ep^{-1}(2\underline\alpha - \overline\alpha - \al_0), \ep^{-1}(2\overline\alpha- \underline\alpha - \al_0)\right)$. Here $A(\tau,\ep)$ is a continuous function such that $\underline\al \leq A(\tau)\leq \overline\al$.
This, together with the boundedness of $\|h_\ep\|_{L^\infty(D)}$ (Lemma \ref{lem:1}), one may apply the Harnack inequality to $D \times(-1,1)$ and deduce the following.
\begin{claim}\label{claim101}
There exists  $C>0$ independent of $\alpha_0 \in [\underline\al,\overline\al]$ and $\ep$ such that
$$
\sup_{D \times(-1,1)} U_\ep(x,\tau) \leq C \inf_{D \times(-1,1)} U_\ep(x,\tau).
$$
\end{claim}
Next, we apply elliptic $L^p$ estimates to $U_\ep$ in $D \times(-1,1)$, so that
\begin{equation}\label{eq:102}
\sup_{x \in D} \left[ \left| U_{\ep,\tau}(x,0)\right| + \left| \nabla_x U_\ep(x,0)\right|\right] \leq C \left\|U_\ep \right\|_{L^p(D \times(-1,1))} \leq C\sup_{D \times(-1,1)} U_\ep(x,\tau).
\end{equation}
In view of Claim \ref{claim101}, we deduce for any $x \in D$,
$$
\left| U_{\ep,\tau}(x,0)\right| + \left| \nabla_x U_\ep(x,0)\right| \leq C \inf_{D \times(-1, 1)} U_\ep \leq CU_\ep(x,0).
$$
 i.e.
$
\ep \left| u_{\ep,\al}(x,\al_0)\right| + \left| \nabla_x u_\ep(x,\al_0)\right| \leq C u_\ep(x,\al_0)
$
for all $x \in D$. Since $C$ is independent of $x$, $\al_0$ and $\ep$, this proves Proposition \ref{prop:PS}.
\end{proof}
%Our alternative approach in this paper, despite its technicality, provides more quantitative
% estimates of the solution profile of $u_\ep$ and  paves the way for the proof of asymptotic stability and uniqueness of $u_\ep$, which will be presented in a subsequent paper.

\section{Two eigenvalue problems} %Preliminaries}
\label{sec:4}

%We fix in the rest of the proof, by Lemmas \ref{lem:1} and \ref{lem:2},  a sequence $ \epsilon_j \to 0$ and a non-constant function $h_0 \in L^\infty(D)$ so that $h_{\epsilon_j} \rightharpoonup h_0$ weakly in $L^p(D)$ for all $p>1$. We will show in the Section \ref{subsect:hatu}
%that $h_0(x) = \theta_{\underline{\alpha}}(x) - m(x)$. The uniqueness of the weak limit would then imply that the convergence $h_\epsilon \rightharpoonup h_0$ holds independent of sequences of $\epsilon_j \to 0$.

\subsection{An Auxiliary Eigenvalue Problem}\label{subsect:4.1}

%Let $h_\epsilon$ be given in \eqref{eq:h} for $\epsilon= \epsilon_k \to 0$ and $h_0$ be the given weak limit of $h_\epsilon$. We c
Consider, for each $\alpha>0$ and $\epsilon> 0$ the  eigenvalue problem (recall $h_\ep(x) = \hat{u}_\ep(x) - m(x)$)
\begin{equation}\label{eq:psi}
\left\{
\begin{array}{ll}
-\alpha {\Delta} \psi + h_\epsilon \psi = \sigma \psi \quad \text{ in }D,\\
\frac{\partial\psi}{\partial n} = 0 \quad \text{ on }\partial D \quad \text{ and }\quad \int_D \psi^2\,dx =
 \int_D \theta^2_{\underline\alpha}\,dx,
\end{array}
\right.
\end{equation}
and denote the principal eigenvalue and positive eigenfunction by $\sigma_\epsilon(\alpha)$ and $\psi_\epsilon(x,\alpha)$, respectively. At this point, we have not shown
how the two eigenvalue problems \eqref{eq:psi} and \eqref{eq:psistar} are related yet.

%For each $\epsilon>0$, consider the following mapping $F_\epsilon: W_\mathcal{N}^{2,p}(D)\times \mathbb{R} \times \mathbb{R}_+ \to L^p(D) \times \mathbb{R}$, given by
%$$
%F_\epsilon(\psi, \sigma, \alpha) = (\alpha \Delta_x \psi - h_\epsilon \psi + \sigma \psi, \int_D \psi\,dx- \int_D \theta^2_{\underline\alpha}\,dx),
%$$
%where $p > N$ and
%$
%W_\mathcal{N}^{2,p}(D) = \{\psi \in W^{2,p}(D): \frac{\partial \psi}{\partial n} = 0 \text{ on }\partial D\}$. Then for each $\alpha>0$, the principal eigenpair $(\psi_\epsilon(\cdot,\alpha), \sigma_\epsilon(\alpha))$ satisfies
%$$
%F_\epsilon(\psi_\epsilon(\cdot,\alpha), \sigma_\epsilon(\alpha), \alpha) = (0,0),
%$$ and by the arguments in the proof of \cite[Lemma 2.1]{CC},  $D_{\psi,\sigma} F_\epsilon (\psi_\epsilon(\cdot,\alpha), \sigma_\epsilon(\alpha), \alpha)$ is an invertible operator from $W^{2,p}_\mathcal{N} \times \mathbb{R}$ to $L^p(D) \times \mathbb{R}$. Hence the implicit function theorem applies and shows that for each $\epsilon>0$, $\alpha \mapsto (\psi_\epsilon(\cdot,\alpha), \sigma_\epsilon)$ is a smooth mapping from $\mathbb{R}_+$ to $W^{2,p}_\mathcal{N}(D) \times \mathbb{R}$.

For each $\epsilon>0$, $\sigma_\epsilon(\alpha)$ is a smooth function of $\alpha>0$ (Proposition \ref{prop:E20}(ii)), and it
has a  Taylor expansion at $\alpha = \underline\alpha$: % $\sigma_\epsilon(\alpha)$ at $\alpha = \underline\alpha$:
\begin{equation}\label{eq:taylorepsilon}
%\begin{array}
\sigma_\epsilon(\alpha)= \sigma_{0,\epsilon} + \sigma_{1,\epsilon}(\alpha - \underline\alpha) + \sigma_{2,\epsilon}(\alpha - \underline\alpha)^2 + O((\alpha - \underline\alpha)^3),
\end{equation}
where $\sigma_{0,\epsilon} = \sigma_\epsilon(\underline\alpha)$ and $\sigma_{k,\epsilon}=\frac{\partial^k}{\partial \alpha^k} \sigma_\epsilon(\underline\alpha)$.
%Next, denote the principal eigenvalue and
% positive eigenfunction by $\sigma_0(\alpha)$ and $\psi_0(x,\alpha)$ for
%\begin{equation}\label{eq:psi0}
%\left\{
%\begin{array}{ll}
%-\alpha {\Delta} \psi + h_0 \psi = \sigma \psi \quad \text{ in }D,\\
%\frac{\partial\psi}{\partial n} = 0 \quad \text{ on }\partial D \quad \text{ and }\quad \int_D \psi^2\,dx =  \int_D \theta^2_{\underline\alpha}\,dx.
%\end{array}
%\right.
%\end{equation}
%Then we may similarly expand $\sigma_0(\alpha)$ at $\alpha=\underline\alpha$ as follows:
%\begin{equation}\label{eq:taylorzero}\sigma_0(\alpha)= \sigma_{0,0} + \sigma_{1,0}(\alpha - \underline\alpha) + \sigma_{2,0}(\alpha - \underline\alpha)^2 + O((\alpha - \underline\alpha)^3),
%\end{equation}
%where $\sigma_{0,0} = \sigma_0(\underline\alpha)$ and $\sigma_{1,0}=\frac{\partial}{\partial \alpha} \sigma_0(\underline\alpha)$.

\begin{lemma}\label{lem:4.1} Let $\sigma_{\epsilon}$
%, $\sigma_{k,0}$,
and $\psi_{\epsilon}$
%and $\psi_0$
be given as above.
\begin{itemize}

\item[{\rm (i)}] For each $k\geq 0$, $\frac{\partial^k}{\partial \al^k} \sigma_\ep (\al)$ is bounded uniformly in $\epsilon>0$ and $\alpha \in [\underline\al, \overline\al]$.

\item[{\rm (ii)}] For each $k \geq 0$ and $p>1$, $\frac{\partial^k}{\partial \al^k}\psi_{\epsilon}(\,\cdot\,,\alpha)$  is bounded in ${W^{2,p}(D)}$ (and hence $C(\overline D)$) uniformly in  $\ep> 0$ and $\alpha \in [\underline\al, \overline\al]$.

\item[{\rm (iii)}] There exists $c_0>0$ such that
$$
\liminf_{\ep \to 0} \frac{\partial \sigma_\ep}{\partial \al}(\alpha) \geq c_0 >0 \quad \text{ for all }\al \in [\underline\al,\overline\al].
$$
In particular, $\displaystyle \liminf_{\ep \to 0} \sigma_{1,\ep} = \liminf_{\ep \to 0} \frac{\partial \sigma_\ep}{\partial \al}(\underline\alpha) >0$.

\item[{\rm (iv)}] There exist positive  constants $c_1, c_2$ such that for all $\epsilon>0$,
$$c_1 \leq \psi_\epsilon(x,\alpha) \leq c_2\quad \text{ for all } x \in D\text{ and }\alpha \in [\underline\al, \overline\al].$$

\end{itemize}
\end{lemma}

\begin{corollary}\label{rmk4.1b}
There exists $C>0$ independent of $\ep$ such that
$$
\left\|\frac{\psi_{\ep,\al}}{\psi_\ep}\right\|_{L^\infty(D\times(\underline\al , \overline\al))} + \left\|\frac{\psi_{\ep,\al\al}}{\psi_\ep}\right\|_{L^\infty(D\times(\underline\al , \overline\al))}  \leq C.
$$
\end{corollary}

\begin{proof}[Proof of Lemma \ref{lem:4.1}] By the uniform boundedness of $\|h_\ep\|_{L^\infty(D)}$ in $\ep$ (Lemma \ref{lem:1}),
assertions  (i) and (ii) follow from Proposition \ref{prop:E30}(i).
To show (iii), it suffices to show, given any sequence $\ep_j \to 0$, and $\al_j \to \al_0 \in [\underline\al, \overline\al]$, $ %\displaystyle
\liminf_{j \to \infty}  \frac{\partial}{\partial \al} \sigma_{\ep_j}(\al_j) >0$. By Lemma \ref{lem:1}, we may assume without loss
of generality that for some $h_0 \in L^\infty(D)$, $h_{\ep_j} \rightharpoonup h_0$ weakly in $L^p(D)$ for all $p>1$. Then, in the notation of Appendix \ref{sec:B},  Proposition \ref{prop:E30}(ii) implies that
$$
\frac{\partial }{\partial \al}\sigma_{\ep_j}(\al_j) = \frac{\partial}{\partial \al} \lambda_1(\al_j,h_{\ep_j}) \to  \frac{\partial}{\partial \al} \lambda_1(\al_0,h_{0}).
$$
Since $h_0$ is non-constant (Lemma \ref{lem:2}), Proposition \ref{prop:E20} implies that the last expression is positive. This proves (iii).

For (iv),
%
%\end{proof}
%\begin{corollary}\label{cor:111}
%There exist positive  constants $c_1, c_2$ such that for all $\epsilon>0$, $c_1 \leq \psi_\epsilon(x,\alpha) \leq c_2$ for all $x \in D$ and $\alpha \in [\underline\al, \overline\al]$.
%\end{corollary}
%\begin{proof}
suppose that along a sequence $\epsilon_j \to 0$ and $\al_j \to \al_0>0$,  either $\inf_{D} \psi_{\epsilon_j}(x,\al_j) \to 0$ or $\sup_{D} \psi_{\epsilon_j}(x,\al_j)  \to \infty$. By the uniform boundedness of $\|h_\epsilon\|_{L^\infty(D)}$  (Lemma \ref{lem:1}), we may assume without loss that $h_{\epsilon_j}$ converges weakly in $L^p(D)$ for all $p>1$.
Hence by Proposition \ref{prop:E30}(ii), $\psi_{\ep_j} = \vphi_1(\,\cdot\,;\al_j,h_{\ep_j}) $ converges to  $\vphi_1(\,\cdot\,; \al_0, h_0)$ uniformly in $D$, and the latter is a strictly positive function in  $C(\overline D)$. This is a contradiction, and proves (iv).
%
%
%Since $\psi_0(x,\alpha) >0$ on $\overline D \times [\underline\al, \overline\al]$, this is a contradiction.
\end{proof}

\subsection{A Transformed Problem}

By the fact that $u_\epsilon(x,\alpha)$ is the principal eigenfunction, with zero
as the corresponding principal eigenvalue, of the problem
\begin{equation}\label{eq:ep1}
\left\{
\begin{array}{ll}
-\alpha {\Delta} \phi - \epsilon^2 \phi_{\alpha\alpha} + h_\epsilon(x) \phi =0 \quad \text{ in }\Omega = D \times (\underline\alpha, \overline\alpha),\\
\frac{\partial}{\partial n}\phi = 0 \quad \text{ on }\partial D \times(\underline\alpha,\overline\alpha),\quad \text{ and }\quad \phi_\alpha =0 \quad \text{ on }D \times \{\underline\alpha,\overline\alpha\},
\end{array}
\right.
\end{equation}
we have the following variational characterization
\begin{equation}\label{eq:var}
0 = \inf_{\tiny \begin{array}{cc} \phi \in H^1(\Omega)\\ \int_\Omega \phi^2=1\end{array}} J_\epsilon [\phi],
\end{equation}
where
\begin{equation}\label{eq:functional}
J_\epsilon[\phi] = \int_\Omega \left[\alpha |\nabla_x\phi|^2 + \epsilon^2 |\phi_\alpha|^2 + h_\epsilon \phi^2\right].
\end{equation}

Define
\begin{equation}\label{eq:s}
s_\epsilon = (\overline\alpha - \underline\alpha)/\epsilon^{2/3},
\end{equation}
and
\begin{equation}\label{eq:wwwww}
w_\epsilon(x,s):= u_\epsilon(x,\underline\alpha + \epsilon^{2/3}s)/\psi_\epsilon(x,\underline\alpha + \epsilon^{2/3}s) \quad \text{ for }x \in D, \, 0 \leq s \leq s_\ep,
\end{equation} where $\psi_\epsilon$ is given by $\eqref{eq:psi}$. Then $w_\epsilon$ satisfies
\begin{equation}\label{eq:w}
\left\{
\begin{array}{ll}
- \alpha \nabla_x \cdot (\psi^2_\epsilon \nabla_x w_\epsilon) - \epsilon^{2/3}(\psi^2_\epsilon w_{\epsilon,s})_s + \psi^2_\epsilon \left[ \sigma_\epsilon(\alpha) - \epsilon^2 \frac{\psi_{\epsilon,\alpha\alpha}}{\psi_\epsilon}\right]w_\epsilon   =0\\  \qquad \qquad\qquad\qquad\qquad\qquad\qquad\qquad\quad \qquad\qquad  \text{ in }D \times(0,s_\epsilon),\\
\frac{\partial}{\partial n}w_\epsilon = 0 \qquad \text{ on }\partial D \times(0,s_\epsilon),\\
w_{\epsilon,s} = - \epsilon^{2/3}\frac{\psi_{\epsilon,\alpha}}{\psi_\epsilon} w_{\epsilon} \quad \text{ on }D \times \{0,s_\epsilon\}.
\end{array}
\right.
\end{equation}
%(Note that the terms the square bracket is now independent of $x$, except for an $O(\ep^2)$ term.)
The  corresponding variational characterization can be written as
\begin{equation}\label{eq:var2}
-\sigma_{0,\epsilon} = - \sigma_\epsilon(\underline\alpha) = \inf_{\tiny \phi \in H^1(D \times(0,s_\epsilon))\setminus\{0\}} \frac{\tilde J_\epsilon[\phi]}{\int_D \int_0^{s_\epsilon} \psi^2_\epsilon \phi^2\,ds dx}
\end{equation}
where
\begin{align*}
\tilde J_\epsilon[\phi] =& \int_D \int_0^{s_\epsilon}  \psi^2_\epsilon\left[ \alpha|\nabla_x \phi|^2 + \epsilon^{2/3}|\phi_s|^2 + \left(\sigma_\epsilon(\alpha) - \sigma_\epsilon(\underline\alpha) - \epsilon^2 \frac{\psi_{\epsilon,\alpha\alpha}}{\psi_\epsilon} \right)\phi^2 \right]\,ds dx\\
& + \epsilon^{4/3}\int_D \left[ \psi_\epsilon \psi_{\epsilon,\alpha} \phi^2\right]_{s = 0}^{s_\epsilon}\,dx.
\end{align*}

\begin{lemma}\label{claim:var}
$0\leq -\sigma_{0,\epsilon} \leq O(\epsilon^{2/3})$.
\end{lemma}
\begin{proof}
First,
%By eigenvalue comparison, %zero,
$\sigma_\epsilon(\underline\alpha)$
is the principal eigenvalue (with principal eigenfunction $\tilde\phi(x,\al) = \psi_\epsilon(x,\underline\alpha)$) of
$$
-\underline\alpha {\Delta} \tilde\phi - \epsilon^2 \tilde\phi_{\alpha\alpha} + h_\epsilon(x)\tilde\phi = \sigma\tilde\phi \quad \text{ in }\Omega =D \times (\underline\alpha, \overline\alpha)
$$
subject to homogeneous Neumann boundary condition, with a variation characterization analogous to \eqref{eq:var} and \eqref{eq:functional}. Since the integrand in \eqref{eq:functional} is monotone increasing in $\underline\al \leq \alpha \leq \underline\al$, $\sigma_\epsilon(\underline\al)$ is necessarily less than the principal eigenvalue of \eqref{eq:ep1}, which is zero. This proves $\sigma_{0,\epsilon} = \sigma_\epsilon(\underline\alpha) \leq 0$.

For the upper estimate, we use a  test function $\phi(x,s) = \eta(s)$ for \eqref{eq:var2}, where  $\eta:[0,\infty) \to [0,1]$ satisfies
$$
\eta(s)>0\quad \text{ for }0\leq s\leq 1,\quad \text{ and }\quad \eta(s)=0 \quad \text{ for }s \geq 2.
$$
Then upon using $\nabla_x \eta =0$,
%by testing with $\phi(x,\alpha) = \eta((\alpha - \underline\alpha)\epsilon^{-2/3})$
% in \eqref{eq:var2},  upon using $|\nabla_x \phi|=0$
% and
 \eqref{eq:taylorepsilon}, and
 $$
 \int_D \left[ \psi_\epsilon \psi_{\epsilon,\alpha} \eta^2\right]_{s = 0}^{s_\epsilon}\,dx = \frac{\eta^2(0)}{2} \frac{\partial}{\partial\alpha} \left[\int_D \psi_\epsilon^2 \,dx \right]_{\alpha = \underline\alpha} = 0,
 $$
(since $\eta(s_\epsilon) = 0$ for $\epsilon$ small,
and by normalization we have $\int_D \psi_\epsilon^2(x,\alpha) \,dx \equiv 1
%\int_D \theta_{\underline\alpha}(x)\,dx
$ for all $\alpha$), we obtain from \eqref{eq:var2} that
%$$
%-\sigma_{0,\epsilon} \leq \frac{\int_\Omega \psi^2_\epsilon \left[ \epsilon^2 |\phi_\alpha|^2 + (\sigma_{1,\epsilon}(\alpha - \underline\alpha) + \sigma_{2,\epsilon}(\alpha - \underline\alpha)^2 + O(\epsilon^2))\phi^2\right]}{\int_\Omega \psi^2_\epsilon \phi^2}.
%$$
%Using the change of variables $\alpha = \underline\alpha + \epsilon^{2/3}s$, we obtain
\begin{align*}
-\sigma_{0,\epsilon} \leq \frac{\int_D \int_0^{2} \psi^2_\epsilon \left[ \epsilon^{2/3} |\eta_s|^2 + (\sigma_{1,\epsilon}\epsilon^{2/3}s + \sigma_{2,\epsilon}\epsilon^{4/3}s^2 + O(\epsilon^2))\eta^2\right]dsdx}{\int_D \int_0^{2} \psi^2_\epsilon \eta^2dsdx}.
\end{align*}
The conclusion follows from Lemma \ref{lem:4.1}(iv).
\end{proof}
%%%%%%%%%%%%%%%%%%%%%%%%%%%%%%%%%%%%%%%%%%%%%%%%%%%%
%%%%%%%%%%%%%%%%%%%%%%%%%%%%%%%%%%%%%%%%%%%%%%%%%%%%
%%%%%%%%%%%%%%%%%%%%%%%%%%%%%%%%%%%%%%%%%%%%%%%%%%%%
%%%%%%%%%%%%%%%%%%%%%%%%%%%%%%%%%%%%%%%%%%%%%%%%%%%%
%%%%%%%%%%%%%%%%%%%%%%%%%%%%%%%%%%%%%%%%%%%%%%%%%%%%
%%%%%%%%%%%%%%%%%%%%%%%%%%%%%%%%%%%%%%%%%%%%%%%%%%%%
\section{ Uniform limit of $\hat{u}_\epsilon$. } \label{subsect:hatu}

In this section, we show that $\hat{u}_\epsilon$ converges to $\theta_{\underline\alpha}$  in $C(\overline D)$. In particular, $h_\epsilon \to \theta_{\underline\al}-m$ in $C(\overline D)$.

Recall that $w_\epsilon$ is defined in \eqref{eq:wwwww}.

\begin{lemma}\label{lem::4.3}
For all $\beta >0$, there exists $C>0$ independent of $\ep$, such that
$$
w_\epsilon(x,s) \leq C \epsilon^{-1} e^{-\beta s} \quad \text{ for all }x \in D \text{ and }0 \leq s \leq s_\epsilon,
$$
where $s_\epsilon = (\overline\alpha - \underline\alpha)/\epsilon^{2/3}$.
\end{lemma}
\begin{proof}
First, we derive a rough upper bound of $w_\epsilon$ from Lemma \ref{lem:uinfty}.
\begin{claim}\label{claim:4.3.1}
There exists $C>0$ such that
$$
\sup_{D \times(\underline\alpha, \overline\alpha)} w_\epsilon \leq C \epsilon^{-1}.
$$
\end{claim}

%The claim follows directly from Lemma \ref{lem:uinfty} and the fact that $w_\epsilon= u_\epsilon/ \psi_\epsilon$, with $\psi_\epsilon$ bounded from above and below, from Lemma \ref{lem:4.1}(ii) and that $\psi_0>0$ in $\bar{D}$.

%Choose $x_\epsilon$ and $\alpha_\epsilon$ such that the supremum of $u_\epsilon$ is attained at $(x_\epsilon, \alpha_\epsilon$. Passing to a subsequence, one may assume that $\alpha_\epsilon \to \alpha_0$.
%Next, let $U_\epsilon(x,\tau) = u_\epsilon(x, \alpha_\epsilon + \epsilon\tau/\sqrt{\alpha_0}) / u_\epsilon(x_\epsilon, \alpha_\epsilon)$, then (extending $u_\epsilon$ by reflection at $\alpha = \underline\alpha, \overline\alpha$ if necessary) one can show that $U_\epsilon$ converges locally uniformly to a  non-negative harmonic function on $D \times \mathbb{R}$. By Proposition  \ref{prop:A1}, and the fact that $U_\epsilon(x_\epsilon,0) =1$,
%$U \equiv 1$. Hence
%$$
%\hat{u}_\epsilon(x) = \int_{\underline\alpha}^{\overline\alpha} u_\epsilon(x,\alpha)\,d\alpha \geq \frac{\epsilon}{2} u_\epsilon(x_\epsilon, \alpha_\epsilon)= \frac{\epsilon}{2} \sup_{D \times(\underline\alpha, \overline\alpha)} u_\epsilon
%$$
%for all $\epsilon$ sufficiently small. By Lemma \ref{lem:1}, we deduce that
%$$
%\sup_{D \times(\underline\alpha, \overline\alpha)} u_\epsilon \leq C\epsilon^{-1}.
%$$
By definition, $\sup w_\epsilon \leq  (\sup u_\epsilon)/(\inf \psi_\epsilon)$, and the claim follows from
the upper bound of $u_\epsilon$ (Lemma \ref{lem:uinfty}) and
Lemma \ref{lem:4.1}(iv).
%the uniform convergence of $\psi_\epsilon$ to $\psi_0$, which is strictly positive function on $\bar D$ (Lemma \ref{lem:4.1}(ii)). % Lemma \ref{lem:4.1}(ii) and that $\inf \psi_0 >0$.
% This proves the claim.

 Next, we construct a supersolution to prove the exponential decay.
 \begin{claim}\label{claim:797}
For each $\beta>0$, there exists $s_0>0$ independent of $\epsilon$ such that
$$
\frac{\sigma_\epsilon(\alpha) - \sigma_\epsilon(\underline\alpha)}{\epsilon^{2/3}} + \frac{\sigma_\epsilon(\underline\alpha)}{\epsilon^{2/3}} - \epsilon^{4/3} \frac{\psi_{\epsilon,\alpha\alpha}(x,\alpha)}{\psi_\epsilon(x,\alpha)} \geq 2\beta^2 \quad \text{ for all }\alpha \in [\underline\alpha + \epsilon^{2/3}s_0 ,\overline\alpha].
$$
\end{claim}
To see the claim, we note that since $\sigma_\epsilon$ is monotone increasing  in $\alpha$ (Proposition \ref{prop:E20}(iii)), for $\alpha = \underline\alpha + \epsilon^{2/3}s$ and  $s \geq s_0$,

\begin{align*}
\frac{\sigma_\epsilon(\alpha) - \sigma_\epsilon(\underline\alpha)}{\epsilon^{2/3}} + \frac{\sigma_\epsilon(\underline\alpha)}{\epsilon^{2/3}}
&\geq  \frac{\sigma_\epsilon(\underline \alpha + \epsilon^{2/3}s_0) - \sigma_\epsilon(\underline\alpha)}{\epsilon^{2/3}} + \frac{\sigma_{0,\epsilon}}{\epsilon^{2/3}}.
% \to \sigma_{1,0}s_0 - C
\end{align*}
By \eqref{eq:taylorepsilon} and Lemma \ref{lem:4.1}(i),
 $$
 \liminf_{\ep \to 0}\frac{\sigma_\epsilon(\underline \alpha + \epsilon^{2/3}s_0) - \sigma_\ep(\underline\al)}{\ep^{2/3}}
 \geq (\liminf_{\ep \to 0} \sigma_{1,\ep}) s_0.
 $$
Taking also
Lemma  \ref{claim:var}
 and Corollary \ref{rmk4.1b}
 into account, we conclude that for $\alpha = \underline\alpha + \epsilon^{2/3}s$ and  $s \geq s_0$,
 $$
\liminf_{\ep \to 0}  \left[
\frac{\sigma_\epsilon(\alpha) - \sigma_\epsilon(\underline\alpha)}{\epsilon^{2/3}} + \frac{\sigma_\epsilon(\underline\alpha)}{\epsilon^{2/3}}
 - \ep^{4/3} \frac{\psi_{\ep,\al\al}(x,\al)}{\psi_\ep(x,\al)}
\right]
\geq (\liminf_{\ep \to 0} \sigma_{1,\ep}) s_0 - C.
 $$
 Since $\liminf_{\ep \to 0} \sigma_{1,\ep} >0$ by Lemma \ref{lem:4.1}(iii),  Claim \ref{claim:797} holds by choosing $s_0$ large. % and $a_0$ is some subsequential limit of $-\sigma_\epsilon(\underline\alpha)/\epsilon^{2/3}$.

\begin{claim}\label{claim:supersol}
For each $\beta>0$, there exists $s_0>0$ independent of $\epsilon$ and a supersolution
$$
\overline{W}(x,s):=\left(\sup_{\tiny \begin{array}{cc} x \in D\\ 0 \leq s \leq s_0 \end{array} }  w_{\epsilon} \right) [\exp(-\beta(s-s_0)) + \exp(\beta(s - (3/2)s_\epsilon))],
$$
defined on $D \times (s_0,s_\epsilon)$ such that
\begin{equation}\label{eq:upperW}
\left\{
\begin{array}{ll}
-\frac{\alpha}{\epsilon^{2/3} \psi^2_\epsilon} \nabla_x \cdot (\psi^2_\epsilon \nabla_x \overline{W}) - \frac{1}{\psi^2_\epsilon} [\psi^2_\epsilon \overline{W}_s]_s \\
\quad + \left( \frac{\sigma_\epsilon(\alpha) - \sigma_\epsilon(\underline\alpha)}{\epsilon^{2/3}} + \frac{\sigma_\epsilon(\underline\alpha)}{\epsilon^{2/3}} - \epsilon^{4/3}\frac{\psi_{\epsilon,\alpha\alpha}}{\psi_\epsilon}\right)\overline{W} \geq 0&\textup{ in }D \times (s_0,s_\epsilon),\\
\frac{\partial \overline{W}}{\partial n}=0 &\textup{ on }\partial D \times(s_0,s_\epsilon),\\
\overline{W}(x,s_0) \geq {w}_\epsilon(x,s_0) &\textup{ for } x\in D,\\
\overline{W}_s(x, s_\epsilon) \geq -\epsilon^{2/3} \frac{ \psi_{\epsilon, \alpha}(x,\overline\alpha)}{\psi_\epsilon(x,\overline\alpha)}\overline{W}(x,s_\epsilon) &\textup{ for }x \in D,
\end{array}
\right.
\end{equation}
where $s_\epsilon = (\overline\alpha-\underline\alpha)/\epsilon^{2/3}$.
\end{claim}
To show the differential inequality, note that the term involving derivatives in $x$ vanishes, and that by Claim \ref{claim:797} and Corollary \ref{rmk4.1b},
\begin{align*}
&-\frac{1}{\psi^2_\epsilon}\left( \psi^2_\epsilon \overline{W}_s\right)_s + \left( \frac{\sigma_\epsilon(\alpha) - \sigma_\epsilon(\underline\alpha)}{\epsilon^{2/3}} + \frac{\sigma_\epsilon(\underline\alpha)}{\epsilon^{2/3}} - \epsilon^{4/3}\frac{\psi_{\epsilon,\alpha\alpha}}{\psi_\epsilon}\right)\overline{W}\\
&\geq -\overline{W}_{ss} + o(1)\overline{W}_s + 2\beta^2 \overline{W}\\
&= (-\beta^2 + o(1) \beta + 2\beta^2)\overline{W} \geq 0.
\end{align*}
It remains to check the boundary condition on $D \times \{s_\epsilon\}$, as the rest follows by definition.
Note that $ s_\epsilon \to \infty$ as $\ep \to \infty$, so that $\exp(-\beta (s_\epsilon - s_0)) \ll \exp(-\beta s_\epsilon/2)$. We therefore have
%$$ \frac{\overline W_s(x,s_\epsilon)}{\overline W(x,s_\epsilon)}
% = \frac{ -\beta\exp(-\beta(s_\epsilon-s_0)) +\beta \exp\left( -{\beta}s_\epsilon/2\right)}{\exp(-\beta(s_\epsilon - s_0)) + \exp(-\beta s_\epsilon/2)} \to \beta.
% $$
% We therefore have
 $$
  \frac{\overline W_s(x,s_\epsilon)}{\overline W(x,s_\epsilon)} + \epsilon^{2/3}\frac{\psi_{\epsilon,\alpha}(x,\overline\alpha) }{ \psi_\epsilon(x,\overline\alpha)}
= \frac{-\beta \exp(-\beta(s_\ep - s_0)) + \beta \exp(-\beta s_\ep/2)}{ \exp(-\beta(s_\ep - s_0)) +  \exp(-\beta s_\ep/2)} + O(\ep^{2/3})%\to \beta >0
$$
which converges to a positive constant $\beta$ uniformly for $x \in D$. This proves the claim.

Now, we claim that
\begin{equation}\label{eq:supersolution}
{w}_\epsilon \leq \overline{W}\quad \text{ in }D \times(0, s_\epsilon).
\end{equation}
By definition, it is easy to see that ${w}_\epsilon \leq \overline{W}$ in $D \times[0,s_0]$. That the inequality holds in $D \times (s_0,s_\epsilon)$ is due to the fact that  $\overline{W}$ is a strict positive supersolution of the linear problem \eqref{eq:upperW} with homogeneous Dirichlet boundary condition on $D \times \{s_0\}$, and Neumann condition on the remaining boundary portions. Standard maximum principle applies and shows that the quotient $\frac{\overline{W} - {w}_\epsilon}{\overline{W}}$ is non-negative. (See, e.g. \cite[p. 48]{BNV}.)

Finally, we obtain Lemma \ref{lem::4.3} by combining Claim \ref{claim:4.3.1} and \eqref{eq:supersolution}.
\end{proof}

\begin{lemma}\label{lem::4.4}
Let $v_\epsilon$ be given by \eqref{eq:v0}, then
$$
\sup_{x \in D}\left| v_\epsilon(x) - \underline\alpha \hat{u}_\epsilon(x)\right| \to 0.
$$
\end{lemma}
\begin{proof}
Given $\epsilon$, take $\delta =  \sqrt{\epsilon} \delta_1$, where $\delta_1$ is given by Lemma \ref{lem:1}.
\begin{align*}
\left| v_\epsilon(x) - \underline\alpha \hat{u}_\epsilon(x)\right|
&=  \left| \int_{\underline\alpha}^{\overline\alpha}(\alpha - \underline\alpha) u_\epsilon(x,\alpha)\,d\alpha\right|\\
&\leq  \delta \left|\int_{\underline\alpha}^{\underline\alpha + \delta} u_\epsilon(x,\alpha)\,d\alpha \right| + (\overline\alpha - \underline\alpha) \left|\int^{\overline\alpha}_{\underline\alpha + \delta} u_\epsilon(x,\alpha)\,d\alpha \right|\\
&\leq \delta \hat{u}_\epsilon(x) + C \int_{\underline\al + \delta}^{\overline\al} \ep^{-1} \exp(-\beta (\alpha-\bar\alpha) / \ep^{2/3})\,d\al\\
&\leq \sqrt{\epsilon} + o(1)
\end{align*}
where we have used Lemma \ref{lem:4.1}(iv) and  Lemma \ref{lem::4.3} in the second last inequality. This proves the lemma.
\end{proof}
\begin{proposition}\label{prop:vu}
$\hat{u}_\epsilon \to \theta_{\underline\alpha}$ in $C(\overline D)$, where $\theta_{\underline\alpha}$ is the unique positive solution of \eqref{eq:theta}. In particular, $h_\ep \to \theta_{\underline\al} - m$ in $C(\overline D)$.
\end{proposition}
\begin{proof}
By Remark \ref{rmk:0} and Lemma \ref{lem::4.4}, we deduce that up to a subsequence $\ep_j \to 0$, both $\hat{u}_{\epsilon_j}$ and $v_\epsilon/\underline\alpha$ converges uniformly in $D$ to some $\hat{u}_0 \in W^{2,p}(D)$. We claim that $\hat{u}_0 $ is a (strong and therefore classical) solution of \eqref{eq:theta}, i.e. for each $z(x) \in C^\infty(\bar D)$,
\begin{equation}\label{eq:weaksol}
\underline\alpha \int_D [\hat{u}_0\Delta_x z + \hat{u}_0(m - \hat{u}_0) z ]\,dx - \underline\alpha\int_{\partial D} \hat{u}_0\frac{\partial z}{\partial n} \,dx =0.
\end{equation}
To show \eqref{eq:weaksol}, multiply \eqref{eq:v} by a test function $z(x)$ and integrate by parts, using the Neumann boundary condition of $v_\epsilon$, we obtain
$$
\int_D [v_\epsilon\Delta_x z + \hat{u}_{\epsilon_j}(m - \hat{u}_{\epsilon_j}) z ]\,dx - \int_{\partial D} v_\epsilon\frac{\partial z}{\partial n}\,dx =0.
$$
Then one can pass to the limit to obtain \eqref{eq:weaksol} by invoking Lemma \ref{lem::4.4}. By the lower estimate in Lemma \ref{lem:1}, there exists $\delta_1>0$ such that $\hat{u}_0(x) \geq \delta_1$ for all $x \in D$. Hence $\hat{u}_0$ is the unique positive solution of \eqref{eq:theta},
i.e. $ \hat{u}_{\ep_j} \rightharpoonup \theta_{\underline\al}$ in $C(\overline D)$. Since the limit is independent of subsequences,  we deduce that  that $\hat{u}_\ep \to \theta_{\underline\al}$ as $\epsilon \to 0$ (not just along subsequences $\ep_j \to 0$).
This proves the proposition.
\end{proof}
%In particular, the uniqueness of limit implies that $h_\ep \to h_0  = \theta_{\underline\al} - m$ uniformly on $D$ as $\epsilon \to 0$ (not just for some sequence $\epsilon_k \to 0$) and that  the subsequential limits given at the beginning of subsection   \ref{subsect:4.1}       are in fact $\sigma_0 = \sigma^*$ as well as $\psi_0 = \psi^*$, where $\sigma^*, \psi^*$ are given by \eqref{eq:psistar}.
%
%the convergence of $h_\epsilon$ to $h_0=\theta_{\underline\alpha} - m$ is independent of subsequences, $\sigma_0 = \sigma_\epsilon$, and $\psi_0 = \psi^*$.
\begin{corollary}\label{cor:4.9}
Let $\sigma_\epsilon(\alpha)$ and $\psi_\epsilon(x,\alpha)$ be the principal eigenvalue and eigenfunction of \eqref{eq:psi}, and let $\sigma^*(\alpha)$ and $\psi^*(x,\alpha)$ be those of \eqref{eq:psistar}. Then as $\epsilon \to 0$,
$\sigma_\epsilon \to \sigma^*$ in $C^k([\underline\alpha,\overline\alpha])$ for all $k$, and $\psi_\epsilon(\cdot,\alpha) \to \psi^*(\cdot,\alpha)$  in $C^k([\underline\al, \overline\al]; W^{2,p}(D))$. In particular,
\begin{equation}\label{eq:sigmastar}
\sigma_{0,\epsilon} \to \sigma^*_{0}:= \sigma^*(\underline\alpha)\quad \text{ and }\quad \sigma_{1,\epsilon} \to \sigma^*_{1}:= \frac{\partial \sigma^*}{\partial\alpha} (\underline\alpha)>0.
\end{equation}
\end{corollary}
\begin{proof}
Since now $h_\epsilon  \to h_0  = \theta_{\underline\al} - m$ in $L^\infty(D)$,
the corollary follows from Proposition \ref{prop:E20}(ii). Since $h_0 = \theta_{\underline\alpha} - m$ is non-constant  (Lemma \ref{lem:2}), Proposition \ref{prop:E20}(iv) asserts that  $\sigma^*_{1}>0$.
\end{proof}

\section{Convergence of $w_\epsilon$}\label{sec:6}
%In this subsection, we consider the convergence of the following normalization of the minimizer $w_\epsilon$ of $\tilde J_\epsilon[\phi]$:
Let $w_\epsilon$ be given by \eqref{eq:wwwww},
%\begin{equation}\ref{eq:www}
%w_\epsilon(x,\alpha) = u_\epsilon (x,\alpha) / \psi_\epsilon(x,\alpha),
%\end{equation}
% where $u_\epsilon$ is a positive solution of \eqref{eq:main} and $\psi_\epsilon$ is given by \eqref{eq:psi}.
define the normalized version $\tilde w_\epsilon = \tilde w_\epsilon(x,s)$ on $D\times[0,s_\epsilon]$ by
$$
\tilde{w}_\epsilon(x,s):= \left( \frac{\int_D \theta_{\underline\alpha}^2\,dx}{\int_\Omega \psi_\epsilon^2 w^2_\epsilon\,ds dx}\right)^{1/2}w_\epsilon(x,s )
%
%\frac{w_\epsilon(x,s )}{(\int_\Omega \psi_\epsilon^2 w^2_\epsilon\,ds dx)^{1/2}}
$$
%where $\sigma_{1,\epsilon}$ is given in \eqref{eq:taylorepsilon},
so that %$($with $s_\epsilon$ given in part {\rm (i)}$)$
\begin{equation}\label{eq:norm}
\int_D \int_0^{s_\epsilon} \psi_\epsilon^2(x,\underline\alpha + \epsilon^{2/3}s) \tilde w_\epsilon^2(x,s)\,dsdx = \int_D \theta_{\underline\alpha}^2\,dx>0.
\end{equation}

%%%%%%%%%%%%%%%%%%%%%%%%%%%%%%%%%%%%%%%%%%%%%%%%%%%%
%%%%%%%%%%%%%%%%%%%%%%%%%%%%%%%%%%%%%%%%%%%%%%%%%%%%
%%%%%%%%%%%%%%%%%%%%%%%%%%%%%%%%%%%%%%%%%%%%%%%%%%%%
%%%%%%%%%%%%%%%%%%%%%%%%%%%%%%%%%%%%%%%%%%%%%%%%%%%%
%%%%%%%%%%%%%%%%%%%%%%%%%%%%%%%%%%%%%%%%%%%%%%%%%%%%
%%%%%%%%%%%%%%%%%%%%%%%%%%%%%%%%%%%%%%%%%%%%%%%%%%%%
%%%%%%%%%%%%%%%%%%%%%%%%%%%%%%%%%%%%%%%%%%%%%%%%%%%%
%%%%%%%%%%%%%%%%%%%%%%%%%%%%%%%%%%%%%%%%%%%%%%%%%%%%
%The next results con minimizer $w_\epsilon$ of \eqref{eq:var2}.
\begin{proposition}\label{lem:w}
\begin{itemize}
\item[{\rm (i)}] $\|\tilde{w}_\ep\|_{H^1(D \times (0,s_\ep))}$ is bounded uniformly in $\epsilon$.

%\item[{\rm (i)}] For each $\delta>0$,
%$$
%\int_D \int_{\delta \sigma_{1,\epsilon}\epsilon^{-2/3}}^{(\overline\alpha - \underline\alpha)\sigma_{1,\epsilon}\epsilon^{-2/3}} \psi^2_\epsilon  w^2_\epsilon\,dsdx \leq O(\epsilon^{2/3}),
%$$
%where the integrand is evaluated at $x \in D$ and $\alpha = \underline\alpha + \epsilon^{2/3}s/\sigma_{1,\epsilon}$.

\item[{\rm (ii)}] For any $\beta>0$,  there exists $C_1>0$ such that for all $\epsilon$ sufficiently small,
$$
\tilde w_\epsilon(x,s ) \leq C_1 e^{-\beta s}
$$
for all $0 \leq s \leq s_\epsilon$.

\item[{\rm (iii)}] As $\epsilon \to 0$, $\tilde w_\epsilon(x, s )$ converges
locally uniformly in $\bar D \times [0,+\infty)$ to the unique positive solution of the problem
%$\tilde\eta(x,s) = \tilde\eta(s)$ to
\begin{equation}\label{eq:eta}
\left\{
\begin{array}{ll}
\tilde\eta_{ss} + (\tilde a_0 - \sigma^*_{1} s) \tilde\eta=0 \quad \text{ for }s \geq 0,\\
\tilde\eta_s(0)= 0 = \tilde\eta(+\infty),\quad \int_0^\infty \tilde\eta^2\,ds = 1,
\end{array}\right.
\end{equation}
where
% = \frac{\partial\sigma^*}{\partial\alpha}(\underline\alpha)$,
$\tilde a_0 = (\sigma^*_{1})^{2/3}A_0$,
%$\sigma^*(\alpha)$ is the principal eigenvalue of \eqref{eq:psistar},
with $\sigma^*_{1}$  given by \eqref{eq:sigmastar} and $A_0$ being the absolute value of the first negative root of the derivative of the Airy function.

\end{itemize}
\end{proposition}
Since $\tilde\eta(+\infty) = 0$ and $\tilde{w}_\ep(x,s) \to 0$ as $s \to +\infty$ uniformly in $x \in D$, we have in fact proved the following.
\begin{corollary}\label{cor:6.1b}
$\|\tilde{w}_\ep(x,s) - \tilde\eta(s)\|_{L^\infty(D \times (0,s_\ep))} \to 0$ as $\ep \to 0$.
%Moreover,
%$\int_0^{s_\ep} \tilde{w}_\ep(x,s) \,ds \to \frac{1}{|D|} \int_0^{\infty} \tilde\eta(s)\,ds$ uniformly in $D$.
\end{corollary}

\begin{proof}
By Lemma \ref{claim:var} and the fact that $\tilde w_\epsilon$ is a minimizer of \eqref{eq:var2}, we obtain
\begin{align*}
&\int_D \int_0^{s_\epsilon} \psi^2_\epsilon \left[ \alpha \epsilon^{-2/3} |\nabla_x \tilde w_\epsilon|^2 +  |(\tilde{w}_{\epsilon})_s|^2 + \left( \frac{\sigma_\epsilon(\alpha) - \sigma_\epsilon(\underline\alpha)}{\epsilon^{2/3}} + O(\epsilon^{4/3}) \right)\tilde w^2_\epsilon \right]dsdx\\
& + \epsilon^{2/3}\int_D \left[ \psi_{\epsilon,\alpha}\psi_\epsilon \tilde{w}^2_\epsilon\right]_{s=0}^{s_\epsilon} \,dx \leq C\left( \int_D \int_0^{s_\epsilon} \psi^2_\epsilon \tilde w^2_\epsilon \,dsdx\right) = C.
\end{align*}
%where
%\begin{equation}\label{eq:s}
%s_\epsilon = {(\overline\alpha - \underline\alpha)\epsilon^{-2/3}}.
%\end{equation}
%So that for each $\delta>0$, if $\alpha \in [\underline\alpha + \delta, \overline\alpha]$, then
%$$
%\sigma_\epsilon(\alpha) - \sigma_\epsilon(\underline\alpha) \geq \sigma_\epsilon(\underline\alpha + \delta) - \sigma_\epsilon(\underline\alpha) \to \sigma_0(\underline\alpha + \delta) - \sigma_0(\underline\alpha)>0,
%$$
%by Lemma \ref{lem:4.1}(i) and the strict monotoncity of $\sigma_0$ in $\alpha$, which follows from Lemma \ref{lem:2}. Hence
%$$
%\int_D \int_{\delta\epsilon^{-2/3} \sigma_{1,\epsilon}}^{(\overline\alpha - \underline\alpha)\epsilon^{-2/3} \sigma_{1,\epsilon}} \psi^2_\epsilon \tilde w^2_\epsilon \,dsdx \leq O(\epsilon^{2/3}).
%$$
%And (i) follows from the fact that $\psi_\epsilon \to \psi_0$ in $C^1(\Omega)$ (Lemma \ref{lem:4.1}(ii)) and the latter is strictly positive on $\bar\Omega$ as it is a principal eigenfunction.
%

%For (ii), we observe f
\begin{claim}\label{claim:6.3}%\label{eq:tracethm}
$\displaystyle \left| \int_D \left[ \psi_{\epsilon,\alpha}\psi_\epsilon \tilde{w}_\epsilon^2\right]_{s=0}^{s_\epsilon} \,dx  \right| \leq C\int_D \int_0^{s_\ep} \psi^2_\epsilon \left(  |\nabla_x \tilde{w}_\epsilon|^2 +  |\tilde{w}_{\epsilon,\alpha}|^2 +  \tilde{w}_\epsilon^2\right)\,dsdx.$
\end{claim}
To prove Claim \ref{claim:6.3}, we apply
 the Trace Theorem, so that  there is $C>0$ independent of $\epsilon$ such that
\begin{align*}
\left| \int_D \left[ \psi_{\epsilon,\alpha}\psi_\epsilon \tilde{w}^2_\epsilon\right]_{s=0}^{s_\epsilon} \,dx  \right| &\leq C \int_D \int_0^{s_\ep}\left(  |\nabla_x \tilde{w}_\epsilon|^2 +  |\tilde{w}_{\epsilon,\alpha}|^2 +  \tilde{w}_\epsilon^2\right)\,dsdx\\
&\leq  C\int_D \int_0^{s_\ep} \psi^2_\epsilon \left(  |\nabla_x \tilde{w}_\epsilon|^2 +  |\tilde{w}_{\epsilon,\alpha}|^2 +  \tilde{w}_\epsilon^2\right)\,dsdx,
\end{align*}
where we have used $\|\psi_{\ep,\al}\psi_\ep\|_{L^\infty(D\times(\underline\al,\overline\al))} \leq C$ (Corollary \ref{cor:4.9}) for the first inequality, and Lemma \ref{lem:4.1}(iv) for the second inequality.

From Claim \ref{claim:6.3}, the normalization \eqref{eq:norm}, the  estimate in the beginning of the proof, and the monotonicity of $\sigma_\epsilon(\alpha)$ in $\alpha$ (Proposition \ref{prop:E20}(iv)), we have
$$
(1-C \ep^{2/3})\int_D \int_0^{s_\epsilon} \psi^2_\epsilon \left[\underline\alpha\epsilon^{-2/3} |\nabla_x \tilde w_\epsilon|^2 + |(\tilde{w}_{\epsilon})_s|^2 + \tilde{w}^2_\epsilon \right] \leq C \int_D \int_0^{s_\ep} \psi^2_\ep \tilde{w}^2_\ep\,dsdx = C.
$$
By Lemma \ref{lem:4.1}(iv), we deduce
\begin{equation}
\int_D \int_0^{s_\epsilon}\left[\epsilon^{-2/3} |\nabla_x \tilde w_\epsilon|^2 + |(\tilde{w}_{\epsilon})_s|^2 + \tilde{w}^2_\epsilon \right] \leq C,
\end{equation}
which implies our assertion (i).  Passing
to a subsequence, $\tilde w_\epsilon(x,s)$ converges weakly in
$H^1_{loc}(D \times[0,\infty))$ to some
function $\tilde\eta$. % that is independent
of $x$. %Here we used Lemma \ref{lem:4.1}(ii) and that $\psi_0>0$ in $\bar{D}$.
Moreover, as $\int_D \int_0^{s_\ep} |\nabla_x \tilde{w}_\ep|^2\,dsdx \leq C \epsilon^{2/3}$, it follows that $\nabla_x \tilde \eta = 0$ a.e..

We outline the rest of the proof of Proposition \ref{lem:w}. First we will show (iii) except the normalization condition
\begin{equation}\label{eq:normal}
\int_0^\infty \tilde\eta^2\,ds=1.
\end{equation} Second, we will show the estimate (ii). Finally we will use (ii) to derive \eqref{eq:normal} from   \eqref{eq:norm}, which completes the proof of (iii).
%
%
%We will first show (ii) except the last normalization condition $\int_0^\infty \tilde\eta^2\,ds=1$. These information will then be used to show (i), from which the normalization condition follows by \eqref{eq:norm}.

We claim that $\tilde\eta$ must satisfy
the equation in \eqref{eq:eta}. To see this claim, note that the equation for $\tilde{w}_\epsilon$ is
\begin{equation}\label{eq:ww}
%\begin{align*}
\begin{array}{ll}
0=&-\frac{\alpha}{\epsilon^{2/3}} \nabla_x \cdot (\psi^2_\epsilon \nabla_x \tilde w_\epsilon) -  [\psi^2_\epsilon (\tilde w_\epsilon)_s]_s \\
&\quad + \left( \frac{\sigma_\epsilon(\alpha) - \sigma_\epsilon(\underline\alpha)}{\epsilon^{2/3}} + \frac{\sigma_\epsilon(\underline\alpha)}{\epsilon^{2/3}} - \epsilon^{4/3}\frac{\psi_{\ep,\alpha\alpha}}{\psi_\epsilon}\right)\psi^2_\epsilon\tilde w_\epsilon
\end{array}
\end{equation}
%\end{align*}
We argue via the weak formulation.
\begin{claim}\label{claim:100}
There exists a constant $\bar a_0$ such that  for each test function $z(s)$ that is compactly supported in $[0,\infty)$,
$$
\int_0^{\infty} \left[-z_s \tilde\eta_s + (\bar a_0 - \sigma^*_{1} s) z\tilde\eta\right]\,ds = 0.
$$
In particular, $\tilde\eta$ satisfies the equation $\tilde\eta_{ss} + (\bar a_0 - \sigma^*_{1} s) \tilde\eta =0$ on $(0,\infty)$ and $\tilde\eta_s(0) = 0$.
\end{claim}

Multiplying  \eqref{eq:ww}  by a test function $z = z(s)$, and integrating
 over $x\in D$, we see that the term involving derivatives in $x$ vanishes (by the Neumann boundary condition $\frac{\partial \tilde w_\epsilon}{\partial n}=0$), and obtain
\begin{equation}\label{eq:weak}
0=-z \int_D   [\psi^2_\epsilon (\tilde w_\epsilon)_s]_s\,dx
 + z\int_D\left[ \frac{\sigma_\epsilon(\alpha) - \sigma_\epsilon(\underline\alpha)}{\epsilon^{2/3}} + \frac{\sigma_\epsilon(\underline\alpha)}{\epsilon^{2/3}} - \epsilon^{4/3}\frac{\psi_{\ep,\alpha\alpha}}{\psi_\epsilon}\right] \psi^2_\epsilon \tilde{w}_\epsilon
\,dx.
\end{equation}

Next, integrate  the first term of \eqref{eq:weak} over $s \in [0,s_\epsilon]$, we see that
\begin{align*}
&\quad - \int_0^{s_\epsilon}z \int_D   [\psi^2_\epsilon (\tilde w_\epsilon)_s]_s\,dxds \\ &= \int_0^{s_\epsilon}\int_D z_s \psi^2_\epsilon (\tilde w_\epsilon)_s\,dxds - \left[ z\int_D \psi^2_\epsilon (\tilde{w}_\epsilon)_s\,dx \right]_{s=0}^{s_\epsilon}\\
&= \int_0^{s_\epsilon}\int_D z_s \psi^2_\epsilon (\tilde w_\epsilon)_s\,dxds + \epsilon^{2/3} \left[ z\int_D \psi_\epsilon \psi_{\ep,\al} \tilde{w}_\epsilon\,dx \right]_{s=0}^{s_\epsilon},
\end{align*}
where we have used the boundary condition $(\tilde w_\epsilon)_s = -\epsilon^{2/3} \psi_{\ep,\al} \tilde{w}_\epsilon/\psi_\epsilon$ on $D \times \{0,s_\epsilon\}$. Since $z(s)$ has compact support in $[0,\infty)$,  the boundary term evaluated at $s=s_\epsilon$ vanishes, and the remaining boundary term is of order $O(\epsilon^{2/3})$ (since $\tilde{w}_\epsilon$ is bounded in $H^1(D \times (0,s_\epsilon))$ by assertion (i), and hence bounded in $L^2(D\times \{0\})$ by the Trace Theorem). Hence, we have
\begin{equation}\label{eq:bd}
- \int_0^{s_\epsilon}z \int_D   [\psi^2_\epsilon (\tilde w_\epsilon)_s]_s\,dxds =  \int_0^{s_\epsilon}\int_D z_s \psi^2_\epsilon (\tilde w_\epsilon)_s\,dxds + o(1).
\end{equation}
%the boundary term tends to zero as $\epsilon \to 0$.
Also,
in the support of $z(s)$, $(\psi_\epsilon)^2(x,\underline\alpha + \epsilon^{2/3}s) \to (\psi^*)^2(x,\underline\alpha)$ uniformly, so we may use \eqref{eq:bd} to integrate \eqref{eq:weak} over $s \in [0,s_\epsilon]$ and pass to the limit to get
\begin{equation}\label{eq:part1}
0=\left(\int_D (\psi^*)^2(x,\underline\alpha)\,dx\right) \left[  \int_0^\infty z_s \tilde\eta_s\,ds + \int_0^\infty (\sigma^*_{1}s -\bar a_0) z \tilde\eta\,ds\right]
\end{equation}
where we have used
Corollary \ref{cor:4.9}
%the expansion \eqref{eq:taylorepsilon}
and that  $\bar a_0$ is a subsequential limit of $-\sigma_\epsilon(\underline\alpha)/\epsilon^{2/3}$ (see also Lemma \ref{claim:var}). This proves Claim \ref{claim:100}.
Next, we claim that
\begin{equation}\label{eq:E31}
\int_0^\infty \tilde\eta^2\,ds < +\infty.
\end{equation}
Notice that by normalization of $\tilde w_\epsilon$
(see \eqref{eq:norm}), and the uniform (in $\ep$) positive upper/lower bound of $\psi_\ep$ (Lemma \ref{lem:4.1}(iv)),
%(see Proposition \ref{lem:w}(ii)),
there exists a fixed constant $C_0$ such that for each $M>0$, and for all $\epsilon>0$ sufficiently small,  $\int_0^M \int_D \tilde w_\epsilon^2\,dxds \leq \int_0^{s_\epsilon}\int_D \tilde w_\epsilon^2\,dxds \leq C_0$. Letting $\epsilon \to 0$, $\int_0^M \tilde\eta^2\,ds \leq C_0$
for all $M>0$. i.e. \eqref{eq:E31} holds.
%$\int_0^\infty \tilde\eta^2\,ds < +\infty$.

\begin{claim}\label{claim:E32}
$\tilde\eta$ is a positive solution that satisfies \eqref{eq:eta} with condition \eqref{eq:normal} being replaced by \eqref{eq:E31}.
\end{claim}
By Claim \ref{claim:100}, $\tilde\eta$ satisfies
\begin{equation}\label{eq:E33}
\tilde\eta_{ss} + (\bar{a}_0 - \sigma^*_1 s) \tilde\eta = 0,\quad  \tilde\eta\geq 0 \quad \text{ on }[0,+\infty),\quad \text{and }\quad \tilde\eta_s(0) = 0.
\end{equation}
%By \eqref{eq:norm}, $\tilde\eta \not\equiv 0$. By strong maximum principle, $\tilde\eta >0$ for all $s$.
It remains to show that $\tilde\eta(+\infty) = 0$,
and that the subsequential limit $\bar{a}_0$ must be determined by $\tilde{a}_0$ of the proposition.
By \eqref{eq:E33} and $\tilde\eta >0$, $\tilde\eta_{ss} \geq 0$ for all $s$ sufficiently large. Hence $\tilde\eta(+\infty)$ exists in $[0,+\infty]$. By \eqref{eq:E31}, $\tilde\eta(+\infty) = 0$.

Hence, $\tilde\eta$ is a constant multiple of $\text{Airy}((\sigma^*_{1})^{1/3}s - A_0)$,
where $A_0$ is the absolute value of the first
negative root of the derivative of the Airy function
$\text{Airy}(x)$. In particular
 the subsequential limit $\bar{a}_0$ given in \eqref{eq:part1} is uniquely determined by $\tilde{a}_0 = (\sigma_1^*)^{2/3}A_0$ (i.e. the full limit $\lim_{\epsilon \to 0} \sigma_\ep(\underline\al)/\ep^{2/3}$ exists). This shows  Claim \ref{claim:E32}.
 % and hence
 %assertion (ii) of the proposition, except for the condition \eqref{eq:normal}. Next, we show (iii).
 To finish the proof of (iii) except for \eqref{eq:normal}, it remains to establish the following.
\begin{claim}\label{lem:conv}
$\tilde w_\epsilon(x,s) \to \tilde\eta(s)$ locally uniformly in $\bar D \times [0,\infty)$. In particular, for each $M>0$, $\|\tilde w_\epsilon\|_{L^\infty(D \times[0,M])}$ is bounded uniformly in $\epsilon$.
\end{claim}
It is  enough to show that for each $M>0$,
\begin{equation}\label{eq:lion}
\sup_{s \in [0,M]} \frac{\sup_{x \in D} \tilde{w}_\epsilon(x,s)}{\inf_{x\in D} \tilde{w}_\epsilon(x,s)} \to 1 \quad \text{ as }\epsilon \searrow 0.
\end{equation}
For, assuming \eqref{eq:lion}, one can write
\begin{equation}\label{eq:harnackx}
\tilde{w}_\epsilon(x,s) = \tilde{w}_\epsilon(x_0,s)(1 + \delta_\epsilon(x,s)) \quad \text{ for some }x_0 \in D,
\end{equation}
 where $\delta_\ep(x,s) \to 0$ in  $L^\infty_{loc}(\bar D \times[0,+\infty))$. Now, if we integrate \eqref{eq:harnackx} over $x \in D$,
% and denote it by $\tilde{W}_\ep(s)$,
 then
$$
\tilde{W}_\ep(s) := \frac{1}{|D|} \int_D \tilde{w}_\epsilon(x,s)\,dx = \tilde{w}_\epsilon(x_0,s)(1 + \hat\delta_\ep(s)),
$$
where $\hat\delta_\ep(s) \to 0$ in $L^\infty_{loc}([0,\infty))$.
Since $\tilde{w}_\epsilon$ is bounded in $H^1(D \times (0,s_\epsilon))$, one can easily deduce that $\tilde{W}_\ep(s) \in H^1_{loc}((0,+\infty)) \subset C^{1/2}_{loc}([0,+\infty))$. Therefore, by Arzel\'{a}-Ascoli Theorem,  $\tilde{W}_\ep(s)$ and hence $\tilde{w}_\epsilon(x_0,s)$ converges to $\tilde\eta(s)$ in $C_{loc}([0,\infty))$.
 Finally, \eqref{eq:harnackx} implies that $\tilde{w}_\epsilon(x,s) \to \tilde\eta(s)$ locally uniformly in $D \times [0,+\infty)$.

%To prove Claim \ref{lem:conv},
It remains to show \eqref{eq:lion}
in a similar fashion as in  Claim \ref{claim:1}.
Assume to the contrary that there exists some constant $c_0>1$, $\epsilon = \epsilon_k \to 0$, $s_k \to s_0 <+\infty$, such that
\begin{equation}\label{eq:lions}
\sup_{x \in D} \tilde{w}_\epsilon(x,s_k) \geq c_0\inf_{x \in D} \tilde{w}_\epsilon(x,s_k).
\end{equation}
Similarly as before, we extend $\tilde{w}_\ep$ by reflection on $D \times \{0\}$ so that $\tilde{w}_\ep$ is defined on $D \times(-s_\ep, s_\ep)$, and define
$$
W_k(x,\tau):= \frac{\tilde{w}_\epsilon(x,s_k + \epsilon^{1/3}\tau/\sqrt{\underline\alpha})}{\sup_{x' \in D} \tilde{w}_\epsilon(x',s_k)} \quad \text{ in }D \times \left(
-(s_\epsilon+s_k)\frac{\sqrt{\underline\al}}{\ep^{1/3}}, (s_\ep - s_k) \frac{\sqrt{\underline\al}}{\ep^{1/3}}\right).
$$
%G
Recall that $s_\epsilon$ is defined in \eqref{eq:s}. 
By the equation \eqref{eq:ww} satisfied by $\tilde{w}_\ep$, $W_k$ satisfies
\begin{equation}\label{eq:WWWW}
\left\{
\begin{array}{ll}
-\frac{\alpha }{\underline\al} \nabla_x \cdot (\psi^2_\ep \nabla_x W_k) - (\psi^2_\ep W_{k,\tau})_\tau \\
\qquad \qquad +  \frac{1}{\underline\alpha}\left( \sigma_\ep\left(\underline\al + \ep^{\frac{2}{3}}s_k + \ep \frac{\tau}{\sqrt{\underline\al}}\right) - \ep^2 \frac{\psi_{\ep,\al\al}}{\psi_\ep}\right)\psi^2_\ep W_k=0,
\end{array}
\right.
\end{equation}
in $D \times \left(
-(s_\ep + s_k)\frac{\sqrt{\underline\al}}{\ep^{1/3}}, (s_\ep - s_k) \frac{\sqrt{\underline\al}}{\ep^{1/3}}\right)$, where
$$
\alpha = \alpha(\tau) =
%\left\{
%\begin{array}{ll}
\underline\al + \left|\ep^{\frac{2}{3}}s_k + \ep \frac{\tau}{\sqrt{\underline\al}}\right| %&\text{ when }\ep^{\frac{2}{3}}s_k + \ep \frac{\tau}{\sqrt{\underline\al}} \geq 0,\\
%\underline\al + \ep^{\frac{2}{3}}s_k + \ep \frac{\tau}{\sqrt{\underline\al}}
%\end{array}
%\right.
$$
 and the boundary conditions
$$\left\{
\begin{array}{ll}
\frac{\partial}{\partial n} W_k = 0&  \text{ on } \quad \partial D \times \left(
-(s_\ep + s_k)\frac{\sqrt{\underline\al}}{\ep^{1/3}}, (s_\ep - s_k) \frac{\sqrt{\underline\al}}{\ep^{1/3}}\right)\\
W_{k,\tau} = - \frac{\ep}{\sqrt{\underline\al}} \frac{\psi_{\ep,\al}}{\psi_\ep} W_k &\text{ on } D \times \left\{-(s_\ep + s_k)\frac{\sqrt{\underline\al}}{\ep^{1/3}}, (s_\ep - s_k) \frac{\sqrt{\underline\al}}{\ep^{1/3}}\right\}.
\end{array}\right.
$$
Since $s_\ep \to +\infty$ as $\ep \to 0$ and that $s_k$ remains bounded, we see in particular that the domain of $W_k$ tends to $D \times \mathbb{R}$ as $k \to \infty$. 
\begin{claim}\label{claim:FFF}
For each $M>0$, $\displaystyle \sigma_\ep\left(\underline\al + s_k\ep^{\frac{2}{3}} + \ep \frac{\tau}{\sqrt{\underline\al}}\right) \to 0$ as $\ep \to 0$, uniformly for $\tau \in [-M,M]$.
\end{claim}
By Lemma \ref{lem:4.1}(i), $\sigma_\ep$ are bounded in $C^1([\underline\al, \overline\al])$ uniformly in $\ep$. Hence we may write
$$
\left|\sigma_\ep\left(\underline\al + s_k\ep^{\frac{2}{3}} + \ep \frac{\tau}{\sqrt{\underline\al}}\right)\right| \leq |\sigma_\ep(\underline\al)| + C \left|\ep^{2/3}s_k + \ep \frac{\tau}{\sqrt{\underline\al}}\right|
$$
and conclude that $\sigma_\ep\left(\underline\al + s_k\ep^{\frac{2}{3}} + \ep \frac{\tau}{\sqrt{\underline\al}}\right)$ goes to zero by  Lemma \ref{claim:var}, and boundedness of $s_k, \tau$. This proves Claim \ref{claim:FFF}.

Since the coefficients of \eqref{eq:WWWW} are bounded in $L^\infty_{loc}(\overline D \times \mathbb{R})$ uniformly in $k$, Harnack inequality, and  the normalization condition $\sup_{x \in D} W_k(x,0) = 1$ implies that $W_k$ are bounded in $L^\infty_{loc}(\overline D \times \mathbb{R})$ uniformly in $k$. Hence we may apply elliptic $L^p$ estimates similarly as in Claim \ref{claim:1} to conclude that a subsequence of $W_k$ converges in $L^\infty_{loc}(\overline{D} \times \mathbb{R})$ to a positive solution of
 $(\psi_0(x,\underline\alpha))^{-2} \nabla_x \cdot(\psi_0^2(x, \underline\alpha)  \nabla_x W) + W_{\tau\tau} =0$ in $D \times \mathbb{R}$. (Here we used Claim \ref{claim:FFF}.)
 Now, we apply the following Liouville Theorem, whose proof is exactly analogous to Proposition \ref{prop:A1} and is skipped.
\begin{proposition}
Let $\psi(x)$ be a smooth positive function defined in $\bar D$, then every positive solution $W$ to $\psi^{-2} \nabla_x \cdot(\psi^2 \nabla_x W) + W_{tt} =0$ in $D \times \mathbb{R}$, subject to Neumann boundary condition on $\partial D \times \mathbb{R}$, is necessarily a constant.
\end{proposition}
So that by normalization $\sup_{x \in D} W_k(x,0) = 1$,  $W_k(x,0) \to 1$ uniformly in $D$. This contradicts \eqref{eq:lions} and proves \eqref{eq:lion}. This establishes Claim \ref{lem:conv}.
Claims \ref{lem:conv} and \ref{claim:E32} establish assertion (iii) except for condition \eqref{eq:normal}.

%ENDS CLAIM

%It is standard
%to conclude by elliptic regularity theory that
%the convergence $\tilde w_\epsilon(x,s) \to \tilde\eta(s)$ is in weak
%$W^{2,p}_{loc}(D \times [0,\infty))$ for all $p>1$, so that
%for each $M$, $\|\tilde\eta\|_{L^\infty(D \times[0,M])}$ and hence  $\|\tilde w_\epsilon\|_{L^\infty(D \times[0,M])}$ are bounded independently of $\epsilon$.

Next,  we proceed to show the estimate in (ii). By the preceding argument in the proof of Lemma \ref{lem::4.3},
specifically the construction of supersolution $\overline{W}$ in Claim \ref{claim:supersol}, we can show that for all $\beta >0$, there exists $s_0 >0$
% and $C>0$
 such that
$$
\tilde{w}(x,s) \leq \left(\sup_{\tiny \begin{array}{cc} x \in D\\ 0 \leq s \leq s_0 \end{array} }  \tilde{w}_{\epsilon} \right) [\exp(-\beta(s-s_0)) + \exp(\beta(s - (3/2)s_\epsilon))]
$$
for $x \in D$ and $s \in [s_0,s_\epsilon)$. Then  (ii)  follows from
Claim \ref{lem:conv}, as the expression inside paranthesis is
bounded uniformly in $\epsilon$.
%
% the fact that  $\|\tilde{w}_\epsilon(x,s) -\tilde\eta(s)\|_{L^\infty(D \times[s_0,s_\epsilon])} \to 0$.
We do not repeat the details.

For (iii), it remains to show \eqref{eq:normal}. By assertion (ii), and that
\begin{equation}\label{eq:factt}
\psi_\ep(x,\underline\al + \ep^{2/3}s) \to \psi^*(x,\underline\al)
%= \theta_{\underline\al}(x)
\quad \text{ and }\quad \tilde{w}_\ep(x,s) \to \tilde\eta(s)
\end{equation}
in $L^\infty_{loc}(\bar{D} \times [0,\infty))$ (by Lemma \ref{lem:4.1}(iv) and  Claim \ref{lem:conv} resp.),
we may pass to the limit in \eqref{eq:norm} to obtain
\begin{align*}
\int_D \theta_{\underline\alpha}^2\,dx = \int_D \int_0^{s_\epsilon} \psi^2_\epsilon(x,\underline\alpha + \epsilon^{2/3}s)\tilde w_\epsilon^2(x,s)\,dsdx \to \int_D (\psi^*)^2(x,\underline\alpha)\,dx \int_0^\infty \tilde\eta^2\,ds
%=\int_D \theta_{\underline\alpha}^2\,dx \int_0^\infty \tilde\eta^2\,ds.
\end{align*}
Upon noting that (see Definition \ref{def:2.2}(ii))
\begin{equation}\label{eq:psipsi}
\psi^*(x,\underline\al) = \theta_{\underline\al}(x)\quad \text{ in }D,
\end{equation}
the proof is completed.
%
%This completes the proof of the Proposition \ref{lem:w}.
\end{proof}

\section{Proof of Theorem \ref{thm}}\label{sec:7}

%Finally, we prove Theorem \ref{thm}

\begin{proof}[Proof of Theorem {\rm \ref{thm}}]
First, we note that by
 Proposition \ref{prop:vu},
\begin{equation}\label{eq:1aa}
\epsilon^{2/3} \int_D \int_0^{s_\ep} \psi_\ep w_\ep\,dsdx =  \int_D \int_{\underline\al}^{\overline\al} u_\ep(x, \al)\,d\al dx = \int_D \hat{u}_\ep \,dx \to \int_D \theta_{\underline\al}\,dx
%\epsilon^{2/3}\int_{0}^{s_\epsilon} w_\epsilon \sim \hat{u}_\epsilon \sim  1,
\end{equation}
as $\epsilon \to 0$.
Furthermore, by \eqref{eq:factt}, \eqref{eq:psipsi}
%$\psi^*(x,\underline\al) = \theta_{\underline\al}(x)$ (Definition \ref{def:2.2})
and the estimate of Proposition \ref{lem:w}(ii),
\begin{equation}\label{eq:1bb}
\int_D \int_0^{s_\ep} \psi_\ep \tilde w_\ep\,dsdx \to \int_D \psi^*(x,\underline\al)\,dx \int_0^\infty \tilde\eta(s)\,ds= \int_D \theta_{\underline\al}(x)\,dx \int_0^\infty \tilde\eta(s)\,ds.
% \int_{0}^{s_\epsilon} \tilde w_\epsilon \sim 1.
\end{equation}
By the definition of $w_\epsilon$ and $\tilde{w}_\ep$, there is a function $\Gamma(\epsilon)$ such that
\begin{equation}\label{eq:1cc}
w_\ep(x,s) = \Gamma(\ep) \tilde{w}_\ep(x,s).
\end{equation}
By \eqref{eq:1aa} and \eqref{eq:1bb}, we have
\begin{equation}\label{eq:1dd}
\lim_{\epsilon \to 0} \ep^{2/3} \Gamma(\ep) = \left( \int_0^\infty \tilde\eta\,ds\right)^{-1}.
\end{equation}

Hence, by \eqref{eq:1cc} and Corollary \ref{cor:6.1b},
%Proposition \ref{lem:w}(iii),
$$
\left\| \epsilon^{2/3} w_\ep(x,(\al - \underline\al)/\ep^{2/3}) - \left( \int_0^\infty \tilde\eta\,ds\right)^{-1} \tilde\eta\left(\frac{\al - \underline\al}{\epsilon^{2/3}}\right)\right\|_{L^\infty(\Omega)} \to 0.
$$
By the fact that $\left( \int_0^\infty \tilde\eta\,ds\right)^{-1} \tilde\eta(s) = \eta^*(s)$ where $\eta^*$ is given in Definition \ref{def:2.2}(iii), we also have
$$
\left\| \epsilon^{2/3} w_\ep(x,(\al - \underline\al)/\ep^{2/3}) - \eta^*\left(\frac{\al - \underline\al}{\epsilon^{2/3}}\right)\right\|_{L^\infty(\Omega)} \to 0.
$$
Using Lemma \ref{lem:4.1}(iv), we have
\begin{equation}\label{eq:facttt}
\left\| \epsilon^{2/3} u_\ep(x,\al) - \psi_\ep(x,\al)\eta^*\left(\frac{\al - \underline\al}{\epsilon^{2/3}}\right)\right\|_{L^\infty(\Omega)} \to 0.
\end{equation}
By the fact that  $\eta^*(s) \leq Ce^{-\beta s}$ for some $C, \beta>0$, \eqref{eq:factt} and \eqref{eq:psipsi}, we have
\begin{equation}\label{claim:111}
\left\|(\psi_\epsilon(x,\alpha) - \theta_{\underline\alpha}(x)) \eta^*\left( \frac{\alpha - \underline\alpha}{\epsilon^{2/3}}\right)\right\|_{L^\infty(\Omega)} \to 0.
\end{equation}
And \eqref{eq:thm2} follows from \eqref{eq:facttt} and \eqref{claim:111}.
\end{proof}

\section{Acknowledgement}
 We are grateful to Professor Benoit Perthame for stimulating discussions
and for bring reference \cite{PS} to our attention, and also Professor Fan-Hua Lin for a comment that simplified the proof of Proposition \ref{prop:A1}. KYL thanks the hospitality of National Center for Theoretical Sciences, National Tsing-Hua University, Taiwan, and
the Institute for Mathematical Sciences, Renmin University of China, where parts of this work are completed.

\appendix

\section{Existence Results}\label{sec:A}

In this section we show the existence of positive solution to \eqref{eq:main}. For this purpose, we fix positive parameters $\ep$ and  $\overline \al > \underline\al$, and denote (in this section only) the principal eigenvalue and eigenfunction of the following problem by ${\mu_1} $ and $\phi_1$.
%
%\subsection{Existence result for a related model}
%
%The following problem concerning the evolution of random dispersal is considered in ......
%
%Let $\Omega = D \times(\underline\alpha, \overline\alpha)$, where $D$ is a smooth bounded domain in $\mathbb{R}^N$ and $\underline\alpha < \overline\alpha$ are given positive constants.
%\begin{equation}\label{eq:main}
%\left\{
%\begin{array}{ll}
%\alpha \Delta_x u + \epsilon^2 u_{\alpha\alpha} + u (m(x) - \hat{u})=0 \quad \text{ in }{\Omega},\\
%\frac{\partial u}{\partial n} =0 \quad \text{ on }\partial D \times (\underline\alpha, \overline\alpha),\\
%u_\alpha =0 \quad \text{ on } D \times\{\underline\alpha, \overline\alpha\}.
%\end{array}
%\right.
%\end{equation}
%
%Let $0 < \underline\alpha < \overline\alpha$ be given positive constants, and let  ${\mu_1}$ be the principal eigenvalue of
\begin{equation}\label{eq:pev2}
\left\{
\begin{array}{ll}
\alpha \Delta_x \phi +\epsilon^2 \phi_{\alpha\alpha} + m(x) \phi + \mu \phi =0 \quad \text{ in }{\Omega}:= D \times (\underline\al, \overline\al),\\
\frac{\partial \phi}{\partial n} =0 \quad \text{ on }\partial D \times (\underline\alpha, \overline\alpha), \\
\phi_\alpha =0 \quad \text{ on }D \times \{\underline\alpha, \overline\alpha\}.
\end{array}
\right.
\end{equation}
\begin{theorem}\label{thm:exist2}
If ${\mu_1} \geq 0$, then the equation \eqref{eq:main} has no positive steady-state. If ${\mu_1} <0$, then the equation \eqref{eq:main} has at least one positive steady-state.
\end{theorem}
\begin{proof}
First, we prove the non-existence result. Suppose $\mu_1 \geq 0$ and let $u$ be a non-negative solution of \eqref{eq:main}. Multiply \eqref{eq:main} by the principal eigenfunction $\phi_1$ of \eqref{eq:pev2}, and integrate by parts, we obtain
$$
\mu_1 \int_D \int_{\underline\al}^{\overline\al} \phi_1^2 \,d\al dx  + \int_D \int_{\underline\al}^{\overline\al} u \hat{u} \phi_1 \,d\al dx =0.
$$
Since $\mu_1 \geq 0$, both terms are non-negative, and both must be identically zero. i.e. $u \equiv 0$.

For the existence result, we  consider, for $\tau \in [0,1]$, the positive solutions of
\begin{equation}\label{eq:general'}
\left\{
\begin{array}{ll}
\alpha {\Delta} u + \epsilon^2 (u)_{\alpha\alpha} + (m(x) - \tau\hat u  - (1-\tau)u) u = 0\quad \text{ in } D \times (\underline\alpha, \bar\alpha),\\
\frac{\partial u}{\partial n} = 0 \quad \text{ on }\partial D \times (\underline\alpha , \overline\alpha), \quad
(u)_\alpha = 0 \quad \text{ in }D \times \{\underline\alpha,\overline\alpha\}.
\end{array}
\right.
\end{equation}
Here we recall that $\hat{u} = \int_{\underline\al}^{\overline\al} u(x,\al)\,d\al$. It remains to show the following claim, from which existence of positive solution to \eqref{eq:main} follows by a standard topological degree argument, as the existence of a unique, linearly stable positive solution to \eqref{eq:general'} when $\tau = 0$ is standard.
%, same as the previous subsection.
\begin{claim}\label{claim:e7}
For some $\delta>0$ independent of $\tau \in [0,1]$, any positive solution $u$ of \eqref{eq:general'} satisfies
$$
\delta < \|u\|_{L^1({\Omega})} < 1/\delta.
$$
\end{claim}
For the upper bound, one can integrate \eqref{eq:general'} over ${\Omega}$ to get
%$$
%\int_{\Omega} u( m(x)- \tau \hat{u} - (1-\tau) u) = 0,
%$$
\begin{align*}
 \int_D \hat{u} m \,dx &= \tau  \int_D \hat{u}^2\,dx + (1-\tau) \int_D \int_{\underline\alpha}^{\overline\alpha} u^2 \,d\alpha dx\\
&\geq \left( \tau  + \frac{1-\tau}{\overline\alpha - \underline\alpha} \right) \int_D \hat{u}^2\,dx\\
&\geq c_0 \int_D \hat{u}^2\,dx \\
&\geq \frac{ c_0}{|D|} \left(\int_D \hat{u}\,dx\right)^2 = \frac{ c_0}{|D|} \|u\|_{L^1({\Omega})}^2,
\end{align*}
from which the upper bound follows.
%to get $\int_{\Omega} m u = \int_\Omega u \hat u$. By Cauchy-Schwartz inequality we get
%$$
%\frac{1}{|D|}\left(\int_D \hat{u}\right)^2 \leq \int_D \hat{u}^2  = \int_\Omega u \hat{u} \leq \sup_D m \int_\Omega u = \sup_D m \int_D \hat{u}.
%$$
%So that we have $\|u\|_{L^1(\Omega)} = \int_D \hat{u} \leq |D|(\sup_D m)$. Moreover,

For the lower bound, let $u = v\phi_1$, where $\phi_1>0$ is
the principal eigenfunction corresponding to the principal eigenvalue
${\mu_1} <0$ of \eqref{eq:pev2}. Moreover, if we normalize $\phi_1$ by
$\int_{\Omega} {\phi_1^2}=1$, then $\sup_\Omega \phi_1$ and
$\inf_{\Omega} {\phi_1}$ are  fixed positive  constants independent of $\tau$,
as \eqref{eq:pev2} is independent of $\tau$. The equation for $v$ can be written as
$$
\left\{
\begin{array}{ll}
\alpha\nabla_x \cdot ({\phi_1^2} \nabla_x v) + \epsilon^2 ({\phi_1^2} v_\alpha)_\alpha + {\phi_1^2} v( -\mu_1 - \tau \hat{u}_\epsilon - (1-\tau) u)=0  \quad \text{ in }{\Omega},\\
\frac{\partial v}{\partial n} =0 \quad \text{ on }\partial D \times(\underline\alpha, \overline\alpha),\quad v_\alpha =0 \quad \text{ on }D \times\{\underline\alpha, \overline\alpha\}.
\end{array}\right.
$$
Hence, if we divide by $v$ and integrate by parts, we have
$$
\int_\Omega {\phi_1^2}(\mu_1 + \tau \hat{u} + (1-\tau)u)\,d\alpha dx = \int_{\Omega} {\phi_1^2} \frac{\alpha |\nabla_x v|^2+ \epsilon^2 | v_\alpha|^2}{v^2} \,d\alpha dx >0.
$$
Hence we have
$$
\left(\sup_{\Omega} {\phi_1}\right)^2[\tau (\overline\alpha - \underline\alpha) + (1-\tau)] \|u\|_{L^1({\Omega})} > -\mu_1  \int_{\Omega} {\phi_1^2} \,d\alpha dx = -\mu_1>0.
$$
Since $\mu_1$ and $\sup_{\Omega} {\phi_1}$ are independent of $\tau$, we have $$\|u\|_{L^1({\Omega})} \geq
\frac{-\mu_1}{(\sup_{\Omega} {\phi_1})^2 [\tau (\overline\alpha - \underline\alpha) + (1-\tau)]}.
$$ Since the latter term is bounded  from below uniformly in $\tau \in [0,1]$, the claim is proved.
\end{proof}
\begin{corollary}
If $\int_D m(x)\,dx > 0$, then for any $\epsilon>0$, \eqref{eq:main} has at least one positive solution.
\end{corollary}
\begin{proof}
Divide the equation \eqref{eq:pev2} by the principal eigenfunction $\phi_1$ and integrate by parts over ${\Omega}$, we get
$$
\int_{D} \int_{\underline\al}^{\overline\al} \frac{\alpha |\nabla_x \phi_1|^2 + \ep^2 | \phi_{1,\alpha}|^2}{\phi_1^2}\, d\al dx + \int_{D} \int_{\underline\al}^{\overline\al} (m(x) + \mu_1)\, d\al dx = 0.
$$
Hence for all $\ep >0$,
$$
\mu_1 \leq -\frac{1}{|D|} \int_D m(x)\,dx < 0,
$$
and the existence of positive solution of \eqref{eq:main} follows from Theorem \ref{thm:exist2}.
\end{proof}

\section{Eigenvalue problems with diffusion parameter $\alpha$ and weight function $h(x)$}\label{sec:B}

For each $\al >0$ and $h \in L^\infty(D)$, let $\lambda_1 = \lambda_1(\al,h) \in \mathbb{R}$ and $\vphi(x) = \vphi_1(x;\al,h)$ be the normalized principal eigenvalue and principal eigenfunction of the following problem.
\begin{equation}\label{eq:pev1}
\left\{
\begin{array}{ll}
-\alpha \Delta_x \vphi + h \vphi = \lambda \vphi \quad \text{ in }D,\\
\frac{\partial \vphi}{\partial n} = 0 \quad \text{ on }\partial D,\quad \int_D \vphi^2\,dx = 1.
\end{array}
\right.
\end{equation}
We shall state and prove a number of properties of $\lambda_1$ and $\vphi_1$, and its dependence on the parameters $\alpha$ and $h$, some of which is folklore among specialists.

\begin{proposition}\label{prop:E20}
\begin{itemize}
\item[(i)] For each $\alpha >0$ and $h \in L^\infty(D)$, the problem \eqref{eq:pev1} has a principal eigenvalue $\lambda_1$ which is simple, and the corresponding eigenfunction $\varphi_1$ can be chosen positive and uniquely determined by the constraint $\int_D \vphi_1^2\,dx = 1$.

\item[(ii)] For each $p>1$, the mapping $(\al,h) \mapsto (\lambda_1, \vphi_1(\cdot))$ is smooth from $\mathbb{R}_+ \times L^\infty(D)$ to $\mathbb{R} \times W^{2,p}_\mathcal{N}(D)$, where $W^{2,p}_\mathcal{N}(D) = \{ \phi \in W^{2,p}(D): \frac{\partial \phi}{\partial n} = 0 \text{ on }\partial D \}$.

\item[(iii)] If $h \in L^\infty(D)$ is non-constant, then $\frac{\partial \lambda_1}{\partial \al}(\al,h) >0$ for all $\al >0$.

%\item[(iv)] If $\sup_{j \geq 0}  \|h_j\|_{L^\infty(D)}$ and $h_j \rightharpoonup h_0$ in
%$L^p(D)$ for all $p >1$, then for each $k\geq 0$,
%$$\frac{\partial^k}{\partial \al^k} \lambda_1(\alpha, h_j) \to \frac{\partial^k}{\partial \al^k} \lambda_1(\alpha, h_0)
%\quad \text{ as }j \to \infty$$
%locally uniformly in $\alpha \in \mathbb{R}_+$, and for all $p>1$,
%$$
%\frac{\partial^k}{\partial \al^k} \vphi_1(\,\cdot\,;\alpha, h_j) \rightharpoonup \frac{\partial^k}{\partial \al^k} \vphi_1(\,\cdot\,;\alpha, h_0)
%\quad \text{  weakly in }W^{2,p}(D)\text{ as }j \to \infty
%$$ locally uniformly in $\alpha \in \mathbb{R}_+$.

\end{itemize}
\end{proposition}
\begin{proof}
Part (i) is well-known. See, e.g. \cite[Section 8.12]{GT}. In particular, the principal eigenvalue is given by the variational characterization
\begin{equation}\label{eq:var11}
\lambda_1(\alpha,h) = \inf_{\vphi \in C^1(\overline D)\setminus \{0\}} \frac{\int_D (\alpha |\nabla_x \vphi|^2 + h \vphi^2)\,dx}{ \int_D \vphi^2\,dx}.
\end{equation}
Fix $p>N$ ($N$ being the dimension of $D$). Consider the following mapping $F: W_\mathcal{N}^{2,p}(D)\times \mathbb{R} \times \mathbb{R}_+ \times L^\infty(D) \to L^p(D) \times \mathbb{R}$, given by
$$
F(\vphi, \lambda, \alpha,h) = (\alpha \Delta_x \vphi - h\vphi + \lambda \vphi, \int_D \vphi^2\,dx- 1),
$$
%where $p > N$ and
%$
%W_\mathcal{N}^{2,p}(D) = \{\psi \in W^{2,p}(D): \frac{\partial \psi}{\partial n} = 0 \text{ on }\partial D\}$.
Then for each $\alpha>0$ and $h \in L^\infty(D)$, the principal eigenpair $(\vphi_1(\cdot;\alpha,h), \lambda_1(\alpha,h))$ of \eqref{eq:pev1} satisfies
\begin{equation}\label{eq:FF}
F(\vphi_1(\, \cdot\,; \alpha, h), \lambda_1(\alpha,h), \alpha,h) = (0,0).
\end{equation}
Assertion (ii) follows from the following claim, in view of the Implicit Function Theorem and the smooth dependence of the operator $F$ on $\alpha$ and $h$.
\begin{claim}\label{claim:E11}
For each fixed $\al >0$ and $h \in L^\infty(D)$,
$$
D_{(\vphi,\lambda)} F (\vphi_1(\, \cdot\,; \alpha, h), \lambda_1(\alpha,h), \alpha,h)\,\,\, :\,\,\,  W^{2,p}_\mathcal{N} \times \mathbb{R}\,\, \to \,\,L^p(D) \times \mathbb{R}$$ is a bijection.
\end{claim}
We shall follow the arguments in the proof of \cite[Lemma 2.1]{CC}. Suppose for some $(\Phi, t) \in W^{2,p}_\mathcal{N} \times \mathbb{R}$, $D_{(\vphi, \lambda)}F (\vphi_1(\, \cdot\,; \alpha, h), \lambda_1(\alpha,h), \alpha,h)[\Phi,t] = (0,0)$, i.e.
\begin{equation}\label{eq:E1}
\alpha \Delta_x \Phi - h \Phi + \lambda_1 \Phi +  t \vphi_1 = 0 \quad \text{ in }D,\quad \text{ and }\quad \frac{\partial \Phi}{\partial n} = 0 \quad \text{ on }\partial D,
\end{equation}
and
\begin{equation}\label{eq:E2}
2\int_D \Phi \vphi_1\,dx = 0,
\end{equation}
where  $\lambda_1  = \lambda(\al,h)$ and $\vphi_1 = \vphi(x;\alpha,h)$.  The Fredholm alternative implies that $\int_D t \varphi_1^2\,dx = 0$, i.e. $t = 0$. Hence $\Phi = c \varphi_1$ for some constant $c$ (as $\lambda_1=\lambda_1(\alpha,h)$ is a simple eigenvalue). But then \eqref{eq:E2} implies that $c = 0$. This shows that the kernel of $D_{(\vphi, \lambda)}F (\vphi_1(\, \cdot\,; \alpha, h), \lambda_1(\alpha,h), \alpha,h)$ is trivial. Now let $(f, q) \in L^p(D) \times \mathbb{R}$ be given, we need to solve for $(\Phi, t)$ in
\begin{equation}\label{eq:E3}
\alpha \Delta_x \Phi - h \Phi + \lambda_1 \Phi +  t \vphi_1 = f \quad \text{ in }D,\quad \text{ and }\quad \frac{\partial \Phi}{\partial n} = 0 \quad \text{ on }\partial D,
\end{equation}
and
\begin{equation}\label{eq:E4}
2\int_D \Phi \vphi_1\,dx = q.
\end{equation}
Set $t = (\int_D f \vphi_1\,dx)/(\int_D \vphi_1^2\,dx)$, then $\int_D (f - t\vphi_1) \vphi_1\,dx = 0$, so \eqref{eq:E3} has solution of the form $\Phi = s \vphi_1 + \Phi^\perp$ where $\Phi^\perp \in W^{2,p}_\mathcal{N}(D)$ is unique and satisfies $\int_D \Phi^\perp \vphi_1\,dx = 0$. Finally, if we set $s = q / (2 \int_D \vphi_1^2\,dx)$ then $(\Phi,t)$ solves \eqref{eq:E3} and \eqref{eq:E4}.  This proves Claim \ref{claim:E11}, which implies assertion (ii).

For (iii), we differentiate \eqref{eq:pev1} with respect to $\alpha$, %and denote $\frac{\partial}{\partial\alpha}$ by the subscript $\alpha$,
\begin{equation}\label{eq:E5}
\left\{\begin{array}{ll}
-\alpha \Delta_x \frac{\partial\vphi_{1}}{\partial \al} + h\frac{\partial\vphi_{1}}{\partial \al} - \lambda_1 \frac{\partial\vphi_{1}}{\partial \al} = \Delta_x \vphi_1 + \frac{\partial\lambda_{1}}{\partial \al} \vphi_1 \quad \text{ in }D,\\
\frac{\partial }{\partial n}\frac{\partial\vphi_{1}}{\partial \al} = 0 \quad \text{ on }\partial D,\quad \text{ and }\quad \int_D \frac{\partial\vphi_{1}}{\partial \al} \vphi_1\,dx = 0.
\end{array}\right.
\end{equation}
Multiply \eqref{eq:E5} by $\vphi_1$ and integrate by parts, we have $\frac{\partial\lambda_{1}}{\partial \al} \int_D \vphi_1^2\,dx = \int_D |\nabla_x \vphi_1|^2\,dx$. Since $h(x)$ is non-constant in $x$, $\vphi_1 = \vphi_1(\,\cdot\,;\al,h)$ is non-constant in $x$ and this implies that $\frac{\partial\lambda_{1}}{\partial \al} >0$. This proves (iii).
\end{proof}

First, we show that $\lambda_1$ and $\vphi_1$ are continuous with respect to the weak topology of $\mathbb{R}_+ \times \cap_{p>1}L^p(D)$.

\begin{lemma}\label{lem:EE}
Let $\lambda_1(\alpha,h)$ and $\vphi_1(\,\cdot\,;\alpha,h)$ be the principal eigenpair of \eqref{eq:pev1}.
\begin{itemize}
\item[(i)] For each $p>1$,  there exists $C'_0 = C'_0(p,M,\underline\al, \overline\al, \partial D)$ such that
$$
|\lambda_1(\al,h)| + \left\|\vphi_1(\,\cdot\,;\al,h)\right\|_{W^{2,p}(D)} \leq C'_0
$$
provided $\al \in [\underline\al, \overline\al]$ and $\|h\|_{L^\infty(D)} \leq M$.

\item[(ii)] If  $\al_j \to \al_0 \in [\underline\al, \overline\al]$, $\sup_{j \geq 0} \|h_j\|_{L^\infty(D)} < +\infty$ and $h_j \rightharpoonup h_0$ in $L^p(D)$ for all $p>1$, then as $j \to \infty$, $\lambda_1(\al_j,h_j) \to \lambda_1(\al_0,h_0)$ and $\vphi_1(\,\cdot\,;\al_j,h_j) \rightharpoonup \vphi_1(\,\cdot\,; \al_0,h_0)$ weakly in $W^{2,p}(D)$ for all $p>1$.
\end{itemize}
\end{lemma}
\begin{proof}
By \eqref{eq:var11},  $\lambda_{1}:= \lambda_1(\alpha, h) $ forms a bounded sequence in $[- \|h\|_{L^\infty(D)} , \|h\|_{L^\infty(D)}]$. The $L^p$ estimate (for $p>N$) applied to \eqref{eq:pev1} and interpolation inequality together imply
$$
\|\vphi_1\|_{W^{2,p}(D)} \leq C \| \vphi_1 \|_{L^p(D)} \leq \frac{1}{2} \|\vphi_1\|_{W^{2,p}(D)}  + C  \| \vphi_1 \|_{L^2(D)},
$$
where $C$ is a generic constant, depending on $\|h\|_{L^\infty(D)}, \underline\al, \overline\al$ and the domain $D$ that changes from line to line. This proves (i).

For (ii), let $\alpha_j \to \alpha_0 \in [\underline\al, \overline\al]$ and  $h_j$ be a uniformly bounded sequence in $L^\infty(D)$ and $h_j \rightharpoonup h_0$ weakly in $L^p(D)$. Denote $\lambda_{1,j} = \lambda_1(\alpha_j, h_j)$ and $\vphi_{1,j} = \vphi_1(\,\cdot\,;\al_j,h_j)$.
By assertion (i),  there are subsequences $\lambda_{1,j'}$ and $\vphi_{1,j'}$ such that
$\lambda_{1}(\alpha_j,h_j) \to \tilde\lambda $ and  $ \vphi_{1}(\,\cdot\,; \al_j,h_j) \rightharpoonup \tilde\vphi$ weakly in $W^{2,p}(D)$, for some $\tilde\lambda \in \mathbb{R}$ and $\tilde\vphi \in W^{2,p}(D)$. Take $\alpha = \alpha_{j'}$, $h = h_{j'}$ in \eqref{eq:pev1}, and pass to the weak limit $j' \to \infty$, we deduce
$$
\left\{
\begin{array}{ll}
-\alpha_0 \Delta_x \tilde\vphi + h_0 \tilde\vphi = \tilde\lambda \tilde\vphi \quad \text{ in }D,\\
\frac{\partial \tilde\vphi}{\partial n} = 0 \quad \text{ on }\partial D,\quad \int_D \tilde\vphi^2\,dx = 1.
\end{array}
\right.
$$
Hence $(\tilde\vphi, \tilde\lambda)$ is an eigenpair of \eqref{eq:pev1} when $\alpha = \alpha_0$, $h = h_0$ and such that $\tilde\vphi \geq 0$. Moreover, $\tilde\vphi$ is non-trivial, as $\int_D \tilde\vphi^2\,dx = 1$. By uniqueness of principal eigenpair, it follows that $\tilde\lambda = \lambda_1(\alpha_0, h_0)$ and $\tilde\vphi = \vphi_1(\,\cdot\, ; \al_0,h_0)$. Since the limit is independent of subsequence, we deduce that the full sequence $\lambda_1(\al_j,h_j) \to \lambda(\al_0, h_0)$ and $ \vphi_1(\,\cdot\,;\al_j,h_j) \rightharpoonup \vphi_1(\,\cdot\,; \al_0, h_0)$ weakly in $W^{2,p}(D)$. This proves the assertion (ii).\end{proof}

Next, we show the following uniform estimate of $(D_{(\vphi,\lambda)}F)^{-1}$.
\begin{lemma}
\label{claim:E6} %Denote $\lambda_1 = \lambda(\al,h)$ and $\vphi_1 = \vphi_1(x;\al,h)$.
There exists $C_2= C_2(M, \underline\al, \overline\al,D)$ such that for any $\alpha \in [\underline\al, \overline\al]$ and $\|h \|_{L^\infty(D)}\leq M$, if
$$
D_{(\vphi, \lambda)}F (\vphi_1(\,\cdot\,;\al,h), \lambda_1(\alpha,h), \alpha, h) [ \Phi, t] = (f, q),
$$
i.e. \eqref{eq:E3} and \eqref{eq:E4} hold with $\lambda_1 = \lambda(\al,h)$ and $\vphi_1 = \vphi_1(x;\al,h)$, then
\begin{equation}\label{eq:E7}
|t| + \|\Phi\|_{W^{2,p}(D)} \leq C_2 (|q| + \|f\|_{L^p(D)}).
\end{equation}
%Moreover, if
%$
%f_j \rightharpoonup f_0 \quad  h_j \rightharpoonup h_0$
%in $L^p(D)$ for some $p>1$, $q_j \to q_0$ and $\alpha_j \to \al_0 >0$, then $t_j \to t_0$ and $\Phi_j \rightharpoonup \Phi_0$ in $L^p(D)$, where
%\begin{equation}\label{eq:E51}
%(\Phi_j,t_j) = [D_{(\vphi, \lambda)}F (\vphi_1(\,\cdot\,;\al,h), \lambda_1(\alpha,h), \alpha, h)]^{-1} (f_j, q_j)%\quad \text{ for }j\geq 0.
%\end{equation}
%for $j \geq 0$.
\end{lemma}
\begin{proof}
Let $M>0$ be given. Suppose to the contrary that there are $\alpha_j \in [\underline\al, \overline\al]$, $h_j$, $\Phi_j$, $t_j$, $q_j$, $f_j$ such that
\begin{equation}\label{eq:E9}
\sup_j \|h_j\|_{L^\infty(D)}  \leq M,\quad |t_j| + \|\Phi_j\|_{W^{2,p}(D)} \to \infty, \quad |q_j| + \|f_j\|_{L^p(D)} \leq 1.
\end{equation}
Without loss of generality, we may assume
$\alpha_j \to \alpha_0 \in  [\underline\al, \overline\al]$ and for some
$h_0 \in L^\infty(D)$, $h_j \rightharpoonup h_0$ weakly in $L^p(D)$.
Denote
$$
\lambda_{1,j} = \lambda_{1}(\al_j,h_j),\quad \text{ and }\quad \vphi_{1,j} = \vphi_1(\cdot;\al_j,h_j)\quad \text{ for }j \in \mathbb{N} \cup\{0\}.
$$
The above arguments ensure that
$$\Phi_j = \Phi_j^\perp + q_j/(2\int_D \vphi^2_{1,j} \,dx) \vphi_{1,j},\quad
\text{  and }\quad t_j = \int_D f_j \vphi_{1,j}\,dx / \int_D \vphi^2_{1,j}\,dx$$
where $\Phi_j^\perp$ is the unique solution of \eqref{eq:E3} subject to the constraint
$\int_D \Phi_j^\perp \vphi_{1,j}\,dx = 0$.
%Here $\vphi_{1,j}(x) = \vphi_1(x;\al_j, h_j)$.
 By the normalization $\int_D \vphi^2_{1,j} \,dx = 1$, we have
\begin{equation}\label{eq:E8}
\Phi_j = \Phi_j^\perp + \frac{q_j}{2} \vphi_{1,j}\quad \text{ and }\quad t_j = \int_D f_j \vphi_{1,j}\,dx .
\end{equation}
Since we have shown that $|\lambda_{1,j}|$ and $\|\varphi_{1,j}\|_{W^{2,p}(D)}$ remain bounded uniformly in $j$,
\eqref{eq:E9} and \eqref{eq:E8} imply that  $\|\Phi^\perp_j\|_{W^{2,p}(D)} \to \infty$. Apply $L^p$ estimate to the equation of $\Phi^\perp_j$, which is
\begin{equation}\label{eq:E10}
\left\{\begin{array}{ll}
\alpha_j \Delta_x \Phi^\perp_j- h_j \Phi^\perp_j + \lambda_{1,j} \Phi^\perp_j  = f_j -  (\int_D f_j \vphi_{1,j}\,dx) \vphi_{1,j} \quad \text{ in }D,\\
\frac{\partial \Phi^\perp_j}{\partial n } = 0 \quad \text{ on }\partial D,\quad \text{ and }\quad \int_D \Phi^\perp_j \vphi_{1,j}\,dx = 0.
\end{array}\right.
\end{equation}
Using the boundedness of $\vphi_{1,j}$ in $W^{2,p}(D)$ and hence in $L^\infty(D)$, we have
\begin{align*}
\|\Phi^\perp_j\|_{W^{2,p}(D)} &\leq C \left[\|\Phi^\perp_j\|_{L^p(D)} + \|f_j -  \left(\int_D f_j \vphi_{1,j}\,dx\right) \vphi_{1,j} \|_{L^p(D)}  \right]\\
&\leq C (\|\Phi^\perp_j\|_{L^\infty(D)} + \|f_j \|_{L^p(D)}  ).
\end{align*}
Hence we must have $\|\Phi^\perp_j\|_{L^\infty(D)} \to \infty$ as $j \to \infty$. Define $\tilde\Phi_j := \Phi^\perp_j / \| \Phi^\perp_j \|_{L^\infty(D)}$, then $\tilde\Phi_j$ satisfies
$$
\left\{\begin{array}{ll}
\alpha_j \Delta_x \tilde\Phi_j - h_j \tilde\Phi_j  + \lambda_{1,j} \tilde\Phi_j  = \tilde{f}_j \quad \text{ in }D,\\
\frac{\partial \tilde\Phi_j }{\partial n } = 0 \quad \text{ on }\partial D,\quad \int_D \tilde\Phi_j  \vphi_{1,j}\,dx = 0, \quad  \text{ and }\quad \sup_D \tilde\Phi_j = 1,
\end{array}\right.
$$
where $\tilde{f}_j = [f_j -  (\int_D f_j \vphi_{1,j}\,dx) \vphi_{1,j} ]/\| \Phi^\perp_j \|_{L^\infty(D)}$ converges to zero in $L^p(D)$ as $j \to \infty$. By $L^p$ estimates,
$\tilde\Phi_j$ is bounded uniformly in $W^{2,p}(D)$. Hence, there is a subsequence $\tilde\Phi_{j'}$ that  converges, weakly in $W^{2,p}(D)$ and strongly in $C^1(\overline D)$, to some function $\tilde\Phi_0$. By normalization $\sup_D\tilde\Phi_0 = \lim_{j'} \left(\sup_D \tilde\Phi_j\right) = 1$. Moreover, $\tilde\Phi_0$ satisfies
(using Lemma \ref{lem:EE}(ii))
%%%%%%%%%%%%%%%%%%%%%%%%%%%%%%%%%%%%%%%%%%%%%
%%%%%%%%%%%%%%%%%%%%%%%%%%%%%%%%%%%%%%%%%%%%%
%%%%%%%%%%%%%%%%%%%%%%%%%%%%%%%%%%%%%%%%%%%%%
%%%%%%%%%%%%%%%%%%%%%%%%%%%%%%%%%%%%%%%%%%%%%
$$
\left\{\begin{array}{ll}
\alpha_0 \Delta_x \tilde\Phi_0 - h_0 \tilde\Phi_0  + \lambda_{1}(\al_0,h_0) \tilde\Phi_0  = 0 \quad \text{ in }D,\\
\frac{\partial \tilde\Phi_0}{\partial n } = 0 \quad \text{ on }\partial D,\quad \text{ and }\quad \int_D \tilde\Phi_0  \vphi_{1,0}\,dx = 0.
\end{array}\right.
$$
Since $\tilde\Phi_0$ is non-negative, Proposition \ref{prop:E20}(i) implies
 $\tilde\Phi_0  = c \vphi_1(\,\cdot\,;\al_0,h_0)= c \vphi_{1,0}$ but the integral constraint implies that $c=0$. i.e. $\tilde\Phi_0 = 0$. This is a contradiction to $\sup_D\tilde\Phi_0=1$. This proves \eqref{eq:E7}.
%
%By \eqref{eq:E7}, $\{(\Phi_j, t_j)\}$ is weakly compact in $W^{2,p}(D) \times \mathbb{R}$.
%It remains to show that the set of subsequential limits of  $\{(\Phi_j, t_j)\}$ is uniquely determined by \eqref{eq:E51} for the case $j=0$.
%For that purposes, we  pass to a subsequence so that $\t_j \to t_0$ and  $\Phi_{j'} \rightharpoonup \Phi_0$ weakly in $L^p(D)$.
%It then follows that
%\eqref{eq:E3} and \eqref{eq:E4} holds for $\al = \al_0$, $\Phi = \Phi_0$, $h =  h_0$, $\lambda
%
% Lemma \ref{claim:E6}.
\end{proof}

\begin{proposition}\label{prop:E30} Let $\lambda_1(\alpha,h)$ and $\vphi_1(\,\cdot\,;\alpha,h)$ be the principal eigenpair of \eqref{eq:pev1}.
\begin{itemize}

\item[(i)] For each $k$, there exists $C'_k = C'_k(M, \underline\al, \overline\al,D)$ such that
\begin{equation}\label{eq:E12}
\sum_{j=0}^k \left|\frac{\partial^j}{\partial \al^j} \lambda_1(\alpha,h)\right| + \left\|\frac{\partial^j}{\partial \al^j} \vphi_1(\,\cdot\,;\alpha,h)\right\|_{W^{2,p}(D)} \leq C'_k
\end{equation}
provided $\alpha \in [\underline\al, \overline\al]$ and $\|h\|_{L^\infty(D)}\leq M$.

\item[(ii)] If $\sup_{j \geq 0}  \|h_j\|_{L^\infty(D)} < +\infty$ and $h_j \rightharpoonup h_0$ in
$L^p(D)$ for all $p >1$, then for each $k \geq 0$,
%for each $k\geq 0$,
$$
\frac{\partial^k}{\partial\al^k}\lambda_1(\cdot,h_j) \to \frac{\partial^k}{\partial\al^k}\lambda_1(\cdot,h_0) \quad \text{ in }C_{loc}([0,\infty))
\quad \text{ as }j \to \infty.
$$
%$$\frac{\partial^k}{\partial \al^k} \lambda_1(\alpha, h_j) \to \frac{\partial^k}{\partial \al^k} \lambda_1(\alpha, h_0)
%\quad \text{ as }j \to \infty$$
%locally uniformly in $\alpha \in \mathbb{R}_+$,
Moreover, given $k \geq 0$, $p>1$, and sequence $\alpha_j \to \alpha_0>0$,
$$
\frac{\partial^k}{\partial \al^k} \vphi_1(\,\cdot\,;\alpha_j, h_j) \rightharpoonup \frac{\partial^k}{\partial \al^k} \vphi_1(\,\cdot\,;\alpha_0, h_0)
\quad \text{  weakly in }W^{2,p}(D)\text{ as }j \to \infty.
$$
\end{itemize}
\end{proposition}
\begin{proof}
%\begin{claim}\label{claim:137}
%For each $k$, there exists $C'_k = C'_k(M, \underline\al, \overline\al,D)$ such that
%\begin{equation}\label{eq:E12}
%\sum_{j=0}^k \left|\frac{\partial^j}{\partial \al^j} \lambda_1(\alpha,h)\right| + \left\|\frac{\partial^j}{\partial \al^j} \vphi_1(\,\cdot\,;\alpha,h)\right\|_{W^{2,p}(D)} \leq C'_k
%\end{equation}
%for all $\alpha \in [\underline\al, \overline\al]$, $\|h\|_{L^\infty(D)}\leq M$.
%\end{claim}
Assertions (i) and (ii) for the case $k=0$ are exactly Lemma \ref{lem:EE}.
%\begin{claim}\label{claim:estvphi}
%There exists $C_1= C_1(\|h\|_{L^\infty(D)}, \underline\al, \overline\al, D)$ such that $$|\lambda_1(\alpha,h)| + \|\vphi_1(\,\cdot\,;\alpha,h)\|_{W^{2,p}(D)} \leq C_1$$ for all $\underline\al \leq \alpha \leq \overline\al.$
%\end{claim}
We first prove assertion (i) for  $k=1$,
differentiate the relation \eqref{eq:FF} with respect to $\alpha$, %and denote $\frac{\partial}{\partial \alpha} = \,\, '$.
\begin{equation}\label{eq:E14}
D_{(\vphi,\lambda)}F \left[ \frac{\partial}{\partial \alpha}\vphi_1(\,\cdot\,;\al,h), \frac{\partial}{\partial \alpha}\lambda_1(\al,h)\right] = - D_\alpha F,
\end{equation}
where the partial derivatives of $F$ are evaluated at $(\varphi_1(\,\cdot\,;\alpha,h), \lambda_1(\alpha,h), \alpha, h )$. By  \eqref{claim:E11}, we may write
$$
(\vphi_1', \lambda_1') = (D_{(\vphi,\lambda)}F)^{-1} [- D_\alpha F] = (D_{(\vphi,\lambda)}F)^{-1} (-\Delta_x \vphi_1, 0)
$$
and deduce by Lemma \ref{claim:E6} that $$|\lambda'_1| + \|\vphi'_1\|_{W^{2,p}(D)} \leq  C_2 \| \Delta_x \vphi_1\|_{L^p(D)}  \leq C_2 \|\vphi_1\|_{W^{2,p}(D)} \leq C.$$
i.e. assertion (i) holds for $k=1$.
We argue inductively for $k > 1$. Suppose (i)  holds for $k = K-1$. We can write
\begin{equation}\label{eq:E17}
%\begin{array}{ll}
D_{(\vphi,\lambda)} F  \left[\frac{\partial^K}{\partial \al^K} \vphi_1(\,\cdot\,;\alpha,h),
\frac{\partial^K}{\partial \al^K} \lambda_1(\alpha,h) \right] = \mathfrak{F}_{K}(\al,h)
%:= D_\alpha^{K}F + \sum_{\ell=1}^{K} \sum_{j=0}^{K-\ell} \sum_{j_1 + ...+ j_k = j} D^{K-j-\ell}_\alpha D^\ell_{(\phi,\lambda)} F [ Z^{(j_1)}, ..., Z^{(j_k)}]
\end{equation}
where
%(denoting $ Z = (\vphi_1,\lambda_1)$ and  $Z^{(j)} = \frac{\partial^j}{\partial \alpha^j}(\vphi_1(\,\cdot\,;\alpha,h),\lambda_1(\al,h))$)
\begin{equation}\label{eq:frakf}
\mathfrak{F}_K(\al,h):= \left( - K \frac{\partial^{K-1}}{\partial \al^{K-1}} (-\Delta_x\vphi_1) - \sum_{k=1}^{K-1} \left(\begin{array}{cc} K \\ k  \end{array}\right) \frac{\partial^{k}}{\partial \al^{k}}\lambda_1 \frac{\partial^{K-k}}{\partial \al^{K-k}}\vphi_1 \,\, ,\,\, 0\right)
%\mathfrak{F}_{K}:= D_\alpha^{K}F - D_{(\vphi,\lambda)}F[Z^{(K)}] + \sum_{\ell=1}^{K} \sum_{j=0}^{K-\ell} \sum_{j_1 + ...+ j_k = j} D^{K-j-\ell}_\alpha D^\ell_{(\phi,\lambda)} F [ Z^{(j_1)}, ..., Z^{(j_k)}] .
\end{equation}
By  the form of $\mathfrak{F}_K$, we can deduce the following result.
%$$
%\|\mathfrak{F}_{K}\|_{L^p(D)} \leq
%$$
%
\begin{claim}\label{claim:mathfrakkk}
$\|\mathfrak{F}_K(\al,h)\|_{L^\infty(D) }\leq C\sum_{k=0}^{K-1}\left(  \left|\frac{\partial^{k}}{\partial \al^{k}}\lambda_1 \right| + \left\|\frac{\partial^{k}}{\partial \al^{k}}\vphi_1\right\|_{W^{2,p}(D)} \right).$ %Moreover, if $
\end{claim}
By the induction assumption (i.e. (i) holds for $k=K-1$) we have $\|\mathfrak{F}_K\|_{L^p(D)} \leq C(M, \underline\al, \overline\al,D).$
Hence we may apply Claim \ref{claim:E6} to \eqref{eq:E17} to conclude the assertion (i) for the case $K$. This induction argument proves (i).

By Lemma \ref{lem:EE}(ii), it remains to prove assertion (ii) for case $k \geq 1$.
%%%%%%%%%%%%%%%%%%%%%%%%%%%%%%%%%%%%%%%%%%%%%%%%%%%%%%
%%%%%%%%%%%%%%%%%%%%%%%%%%%%%%%%%%%%%%%%%%%%%%%%%%%%%%
%%%%%%%%%%%%%%%%%%%%%%%%%%%%%%%%%%%%%%%%%%%%%%%%%%%%%%
Let $\alpha_j \to \alpha_0 \in [\underline\al, \overline\al]$ and  $h_j$ be a uniformly bounded sequence in $L^\infty(D)$ and $h_j \rightharpoonup h_0$ weakly in $L^p(D)$. Denote $\lambda_{1,j} = \lambda_1(\alpha_j, h_j)$ and $\vphi_{1,j} = \vphi_1(\,\cdot\,;\al_j,h_j)$.
By assertion (i),  there are subsequences $\lambda_{1,j'}$ and $\vphi_{1,j'}$ such that for all $k\geq 0$, $\frac{\partial^k}{\partial \al^k} \lambda_{1}(\alpha_{j'},h_{j'}) \to \tilde\lambda_k $ and  $\frac{\partial^k}{\partial \al^k} \vphi_{1}(\,\cdot\,; \al_{j'},h_{j'}) \rightharpoonup \tilde\vphi_k$ weakly in $W^{2,p}(D)$, for some $\tilde\lambda_k \in \mathbb{R}$ and $\tilde\vphi_k \in W^{2,p}(D)$.
Passing to the limit in \eqref{eq:E14}, we deduce that
\begin{equation}\label{eq:E15}
D_{(\vphi,\lambda)}F \left[ \tilde\vphi_1,\tilde\lambda_1\right]= - D_\alpha F,
\end{equation}
where the partial derivatives of $F$ are evaluated at $(\vphi_1(\,\cdot\,;\alpha_0, h_0), \lambda_1(\alpha_0, h_0), \alpha_0, h_0) $. Since we also have
\begin{equation}\label{eq:E16}
D_{(\vphi,\lambda)}F\left[ \frac{\partial}{\partial \alpha} \vphi_1(\,\cdot\,; \alpha_0,h_0) , \frac{\partial}{\partial \alpha} \lambda_1(\al_0,h_0)\right] = - D_\alpha F,
\end{equation}
where the partial derivatives of $F$ are evaluated at $(\vphi_1(\,\cdot\,;\alpha_0, h_0), \lambda_1(\alpha_0, h_0), \alpha_0, h_0) $, we may invert $D_{(\vphi,\lambda)}F$ in both \eqref{eq:E15} and \eqref{eq:E16}, and conclude that
$$
\tilde\vphi_1= \frac{\partial}{\partial \alpha} \vphi_1(\,\cdot\,; \alpha_0,h_0) \quad\text{ and }\quad \tilde\lambda_1= \frac{\partial}{\partial \alpha} \lambda_1(\al_0,h_0).$$
Since the limit is determined independent of the subsequence, we conclude assertion (ii) for the case $k=1$.

Again, we may argue inductively for $k > 1$. Suppose (ii) is proved for $k=1,...,K-1$.  The following can be easily observed from \eqref{eq:frakf}.
\begin{claim}\label{claim:E40}
If assertion (ii) holds for $k=1,...,K-1$, then
$$
\mathfrak{F}_K(\alpha_j,h_j) \rightharpoonup \mathfrak{F}_K(\alpha_0, h_0)%\quad \text{ weakly in }L^p(D).
$$
weakly in $L^p(D)$, where $\mathfrak{F}_K$ is defined in \eqref{eq:frakf}.
\end{claim}
%\noindent The claim is a direct consequence of the fact that, for each $\lambda,\al,h$, $\vphi \mapsto F(\vphi,\lambda,\al,h)$ is a linear function of from $W^{2,p}(D)$ to $L^p(D)$.
%  $\mathfrak{F}_K$, on the right-hand side of \eqref{eq:E17}, is a continuous function of the derivatives of $\vphi_1$ with respect to the
%  weak topology of $W^{2,p}(D)$. This is due to the linearity of $F$ in the variable $\vphi$.
Based on Claim \ref{claim:E40}, and the  assertion (ii) for the cases $k=1,...,K-1$, we may pass to the limit in \eqref{eq:E17}. Together with the uniform boundedness of $[D_{(\vphi,\lambda)}F]^{-1}: L^p(D) \to W^{2,p}(D)$ (Lemma \ref{claim:E6}), this implies $\frac{\partial^K}{\partial \al^K}\lambda_{1}(\al_j,h_j) \to \frac{\partial^K}{\partial \al^K}\lambda_1(\al_0,h_0)$ and
$$
\frac{\partial^K}{\partial \al^K} \vphi_1(\cdot;\al_j,h_j) \rightharpoonup \vphi_1(\cdot;\al_0,h_0)\text{ in }W^{2,p}(D).
$$
Thus assertion (ii) follows by induction on $k$.
%\begin{equation}\label{eq:E17}
%D_{(\vphi,\lambda)} F  \left[ \frac{\partial^{K}}{\partial \al^{K}} \vphi_1,\frac{\partial^K}{\partial \al^{K}} \lambda_1\right] = \mathfrak{F}_{K}
%\end{equation}
%where $\mathfrak{F}_{K} = \mathfrak{F}_K(\left\{\frac{\partial^j}{\partial \al^j} \vphi_1\}_{j=1}^K, \left\{\frac{\partial^j}{\partial \al^j} \lambda_1\right\}_{j=0^K}\right\}$ is continuous with respect to the weak topology of $W^{2,p}(D)]^K \times \mathbb{R}^K$.
%By passing to the (weak) limit in \eqref{eq:E17} and using assertion (ii) for $k=1,...,K$, we can similarly prove (ii) for $k=K+1$.
\end{proof}

\section{Liouville Theorem for Positive Harmonic Functions in Cylinder Domain}\label{sec:C}
We give a proof of the Liouville-type theorem for positive harmonic functions in cylinder domains, since we cannot locate a proper reference for this result.
\begin{proposition}\label{prop:A1}
Let $k \in \mathbb{N}$, $D$ be a bounded smooth domain in $\mathbb{R}^{N}$ and  $u$ be a non-negative harmonic function on $\Omega:= D \times \mathbb{R}^k \subset \mathbb{R}^{N+k}$, so that $\frac{\partial u}{\partial n} = 0$ on $\partial D \times \mathbb{R}^k$. Then $u$ is necessarily a constant.
\end{proposition}
\begin{proof}
Let $x \in D$, $y \in \mathbb{R}^k$ and let
$u(x,y)$ be a non-negative harmonic function on $\Omega = D \times \mathbb{R}^k$, subject to Neumann boundary condition on $\partial D \times \mathbb{R}^k$. By subtracting a positive constant from $u$, we may assume that $\inf_\Omega u = 0$.

 Harnack inequality says that there is a constant $C>1$ such that for all  $y' \in \mathbb{R}^k$, we have
$$
\sup_{x \in D, |y-y'| <2} u \leq C \inf_{x \in D, |y-y'| <2} u.
$$

Define $v(y) = \frac{1}{|D|}\int_D u(x',y)\,dy$, then $v$ is a harmonic function on $\mathbb{R}^k$ and must be equal to a non-negative constant $v_0$. Hence for each $y' \in \mathbb{R}^k$, there exists $x' \in \bar{D}$ such that $u(x',y') = v_0$. It follows that for each $y' \in \mathbb{R}^k$,
$$
v_0 \leq C\inf_{x \in D, |y-y'| <2}u(x,y).
% \quad \text{ for all }y'
%\quad \Rightarrow \quad \inf_{\Omega} u .
$$
Taking infimum in $y' \in \mathbb{R}^k$, it follows that from $\inf_\Omega u = 0$ that $v_0=0$. Hence,
$$
\frac{1}{|D|} \int_D u(x,y)\,dx = v(y) =  v_0 = 0
%\text{ for all }y \in \mathbb{R}^k
$$
for all $y\in \mathbb{R}^k$. i.e. $u \equiv 0$ in $\Omega$.
\end{proof}
\bibliographystyle{amsplain}

\end{document}